\documentclass[11pt,twoside]{amsart}
\usepackage{amsmath,amssymb,amsthm,stmaryrd}
\usepackage{mathrsfs}
\usepackage{nicefrac}
\usepackage{graphicx}
\usepackage[dvipsnames]{xcolor}
\usepackage{hyperref}

\makeatletter
\@namedef{subjclassname@2020}{%
  \textup{2020} Mathematics Subject Classification}
\makeatother

\numberwithin{equation}{section}

\newtheorem{Thm}{Theorem}[section]
\newtheorem{Lem}[Thm]{Lemma}
\newtheorem{slem}[Thm]{Sublemma}
\newtheorem{Pro}[Thm]{Proposition}
\newtheorem{Cor}[Thm]{Corollary}

\newcommand{\R}{{\mathbb R}}
\newcommand{\N}{{\mathbb N}}
\newcommand{\Z}{{\mathbb Z}}

\newcommand{\Hm}{{\mathcal H}}
\newcommand{\cH}{{\mathcal H}}

\newcommand{\bI}{{\mathbf I}}

\newcommand{\M}{{\mathbf M}}
\newcommand{\bN}{{\mathbf N}}

\newcommand{\cB}{{\mathscr B}}

\newcommand{\al}{\alpha}

\newcommand{\del}{\delta}
\newcommand{\Del}{\Delta}
\newcommand{\eps}{\epsilon}
\newcommand{\gam}{\gamma}
\newcommand{\ga}{\gamma}
\newcommand{\Ga}{\Gamma}

\newcommand{\Gam}{\Gamma}

\newcommand{\la}{\lambda}
\newcommand{\lam}{\lambda}

\newcommand{\om}{\omega}
\newcommand{\Om}{\Omega}
\renewcommand{\rho}{\varrho}
\newcommand{\sig}{\sigma}
\newcommand{\Sig}{\Sigma}
\newcommand{\ph}{\varphi}

\newcommand{\area}{\operatorname{area}}

\newcommand{\Jac}{\operatorname{Jac}}
\newcommand{\tr}{\operatorname{tr}}

\newcommand{\length}{\operatorname{length}}
\newcommand{\len}{\operatorname{length}}

\newcommand{\bb}[1]{\llbracket #1\rrbracket} 
\renewcommand{\d}{\partial}

\newcommand{\diam}{\operatorname{diam}}

\newcommand{\cs}{\text{\rm c}} %roman subscript for compact support
\newcommand{\rad}{\operatorname{rad}}
\newcommand{\spt}{\operatorname{spt}}
\newcommand{\CAT}{\operatorname{CAT}}

\newcommand{\sub}{\subset}
\newcommand{\sm}{\setminus}

\newcommand{\vol}{\operatorname{vol}}
\newcommand{\Fillvol}{\operatorname{Fillvol}}
\newcommand{\ic}{{\mathbf c}} %optimal isoperimetric constant

\newtheorem{introthm}{Theorem}

\newtheorem{introcor}[introthm]{Corollary}
\newcommand{\bcor}{\begin{Cor}}
\newcommand{\ecor}{\end{Cor}}
\newcommand{\ben}{\begin{enumerate}}
\newcommand{\bit}{\begin{itemize}}
\newcommand{\blem}{\begin{Lem}}
\newcommand{\bslem}{\begin{slem}}
\newcommand{\bprop}{\begin{Pro}}
\newcommand{\bthm}{\begin{Theorem}}
\newcommand{\een}{\end{enumerate}}
\newcommand{\eit}{\end{itemize}}
\newcommand{\elem}{\end{Lem}}
\newcommand{\eslem}{\end{slem}}
\newcommand{\eprop}{\end{Pro}}
\newcommand{\ethm}{\end{Theorem}}

\begin{document}

\title[An isoperimetric gap theorem in non-positive curvature]{Minimal tetrahedra and an isoperimetric gap theorem in non-positive curvature}
\author{Cornelia Dru\c{t}u}
\address{Mathematical Institute\\University of Oxford\\Andrew Wiles Building\\
Woodstock Road\\Oxford OX2 6GG\\UK}
\email{drutu@maths.ox.ac.uk}
\author{Urs Lang}
\address{Department of Mathematics\\ETH Z\"urich\\R\"amistrasse 101\\8092 Z\"urich\\Switzerland}
\email{lang@math.ethz.ch}
\author{Panos Papasoglu}
\address{Mathematical Institute\\University of Oxford\\Andrew Wiles Building\\
Woodstock Road\\Oxford OX2 6GG\\UK}
\email{papazoglou@maths.ox.ac.uk}
\author{Stephan Stadler}
\address{Max Planck Institute for Mathematics\\Vivatsgasse 7\\53111 Bonn\\Germany}
\email{stadler@mpim-bonn.mpg.de}
\date{\today}
%
%\thanks{Research supported by Swiss National Science Foundation Grant 197090.}
%\subjclass[2020]{53C23, 51F30, 20F67}
%\keywords{Gromov hyperbolicity, non-positive curvature, asymptotic rank, linear isoperimetric inequality}

\begin{abstract} 
We investigate isoperimetric inequalities for Lipschitz 2-spheres in CAT(0) spaces, proving bounds on the volume of efficient null-homotopies. In one dimension lower, it is known that a quadratic inequality with a constant smaller than $\ic_2 = 1/(4\pi)$ -- the optimal constant for the Euclidean plane -- implies that the underlying space is Gromov hyperbolic, and a linear inequality holds. We establish the first analogous gap theorem in higher dimensions: if a proper CAT(0) space satisfies a Euclidean inequality for $2$-spheres with a constant below the sharp threshold $\ic_3 = 1/(6\sqrt{\pi})$, then the space also admits an inequality with an exponent arbitrarily close to $1$.
As a corollary we obtain a similar result for Lipschitz surfaces of higher genus. Towards our main theorem we prove a (non-sharp) Euclidean isoperimetric inequality for null-homotopies of $2$-spheres, apparently missing in the literature. A novelty in our approach is the introduction of minimal tetrahedra, which we demonstrate satisfy a linear inequality. 
\end{abstract}

\maketitle

%%%%%%%%%%%%%%%%%%%%%%%%%%%%%%%%%%%%%%%%%%%%%%%%%%%%%%%%%%%%%%%%%%%%%%%%%%%%%%%

\section{Introduction}

\subsection{Background and main results}

Isoperimetric filling inequalities in dimension $n+1$ control the volume needed in order to fill an $n$-cycle $S$ by an $(n+1)$-chain $V$.
Their specific form is intimately related to the geometry of the underlying metric space. A central role is played by isoperimetric inequalites of the Euclidean type
\[
\vol_{n+1}(V)\leq \mathop{\rm const} \cdot\vol_n(S)^{1+1/n}.
\]
By a fundamental result of Wenger \cite{Wen-EII}, which builds on earlier work of Gromov~\cite{Gro_FRR}, such inequalities hold in particular in all Banach spaces and $\CAT(0)$ spaces. 
(See also~\cite{Schulze} for the optimal inequality for $2$-cycles in Hadamard manifolds.)
It is natural to expect that for $\CAT(0)$ spaces, Euclidean isoperimetric inequalities continue to hold when the classes of admissible cycles and fillings are restricted to Lipschitz spheres and balls, respectively.
However, to our knowledge, for $n > 1$ this is not recorded in the literature, and we will provide the result for $n = 2$ (with a non-optimal constant) in Theorem~\ref{main_3} below.
In the case $n=1$ -- when circles are filled by discs -- much more is known. It has recently been shown that non-positive curvature is equivalent to 
a Euclidean isoperimetric inequality with the sharp constant $\ic_2 := 1/(4\pi)$ \cite{LWcurv, StaWen}. Moreover, a length space
that admits a quadratic isoperimetric inequality for curves with a constant strictly smaller than $\ic_2$ is necessarily Gromov hyperbolic \cite{Wen_sharp, LWY_dehn, StaWen}. These sharp results were predated by the general observation, originally due to Gromov \cite{Gro-HG}, that
length spaces with a subquadratic isoperimetric inequality must in fact satisfy a 
linear isoperimetric inequality~\cite{Ol_sub, Pap_sub, Pap_strong, Bow_sub, Drutu-iso}.

The main result of this paper is a first result of this kind
in higher dimension: a sharp isoperimetric gap theorem for fillings of $2$-spheres by $3$-balls in $\CAT(0)$ spaces.

\begin{introthm}\label{main_1}
For a proper $\CAT(0)$ space $X$, the following are equivalent:
\begin{enumerate}
\item
There exists a constant $c < \ic_3 := 1/(6\sqrt{\pi})$ such that every Lipschitz $2$-sphere $\hat S\subset X$ of large area admits a filling by a Lipschitz $3$-ball $\hat B\subset X$ with volume
\[
\vol(\hat B) \leq c\cdot \area(\hat S)^{3/2}.
\]
\item
For every $\delta > 0$ there exists a constant $C = C(\del)$ such that every Lipschitz $2$-sphere $S\subset X$ extends to a Lipschitz  $3$-ball $B\subset X$ with volume 
\[
\vol(B) \le C \cdot\area(S)^{1+\del}.
\]
\item 
The asymptotic rank of $X$ is at most $2$. 
\end{enumerate}
\end{introthm}

The notion of {\em asymptotic rank}\/ appearing in the last item stems from Gromov's discussion of a variety of related invariants in~\cite[\S{6.B$_2$}]{Gro-AI}, 
and was further investigated in~\cite{Kle, W_erk, Wen_AR, Des, KleL, GolL, LS_2dim}.
For proper $\CAT(0)$ spaces with cocompact isometry group, the asymptotic rank equals 
the maximal $n$ such that $X$ contains an {\em $n$-flat}, that is, an isometric copy of $\R^n$. 
In particular, the isoperimetric inequality in the second item of Theorem~\ref{main_1} holds for the universal cover $X$ of any compact manifold of non-positive curvature provided that $X$ contains no $3$-flat.
For a general $\CAT(0)$ space $X$, the asymptotic rank is at most $n$ if and only if no asymptotic cone of $X$ contains an $(n+1)$-flat, 
and it is at most $1$ if and only if $X$ is Gromov hyperbolic. 
We refer to the beginning of Section~\ref{Sec:proof} for the actual definition in the general case. Notice that the asymptotic rank is a quasi-isometry invariant.

Gromov conjectured that proper cocompact $\CAT(0)$ spaces of asymptotic rank at most $n$  admit linear isoperimetric inequalities 
\[
\vol_{n+1}(V)\leq \mathop{\rm const} \cdot\vol_n(S)
\]
for fillings of $n$-cycles by $(n+1)$-chains (compare~\cite[\S{6.B$_2$}, p.~128, (b)]{Gro-AI}). 
Apart from the (hyperbolic) rank one case (see~\cite{Lan_LII} for $n > 1$),  this is presently known for symmetric spaces or homogeneous Hadamard manifolds (see~\cite[\S{5.D}, (5)(b$_1'$)]{Gro-AI} and~\cite{Leu,Isl}). 
For general $\CAT(0)$ spaces of asymptotic rank at most $n$, where $n \ge 2$, the best known result in this direction is Wenger's sub-Euclidean inequality~\cite{Wen_AR}, 
stating that every $n$-cycle of mass $s$ admits a filling with mass at most $o(s^{1+1/n})$ as $s \to \infty$.
In particular, since the isoperimetric inequalities in dimensions up to the asymptotic rank are at best Euclidean, the isoperimetric profiles detect the asymptotic rank. 
Furthermore, a family of linear isoperimetric inequalities for $n$-cycles with bounded density quotients was deduced in~\cite{GolL}. 
The result is restated in Theorem~\ref{Thm:fillrad} below and provides in addition a bound on the filling radius, which will enter the proof of Theorem~\ref{main_1}. 
These inequalities, as well as those in~\cite{Wen-EII, Wen_AR}, actually refer to the chain complex of metric integral currents in the sense of Ambrosio--Kirchheim, which comprises the Lipschitz singular chains and has favorable compactness properties
(see Section~\ref{Sec:currents} for further information). 
By contrast, Theorem~\ref{main_1} concerns Lipschitz mappings and null-homotopies of spheres, the use of integral currents will thus be limited to the proofs of two intermediate results.

In the terminology of~\cite[Section~6.D]{Grom_Met}, the second assertion in Theorem~\ref{main_1} is referred to as an isoperimetric inequality of {\em infinite rank}. Inequalities of this type are also discussed in~\cite{Pap-CC, Wen_AR}. 
A priori, this assertion is weaker than the conjectured linear bound (possibly $C = C(\del) \to \infty$ as $\del \to 0$), and proving the latter, if correct, seems to require a strategy different from ours. 
However, a linear inequality will be obtained for so-called minimal tetrahedra in Theorem~\ref{main_4} below, which constitutes a crucial ingredient of Theorem~\ref{main_1}. 
Currently we do not know how to adapt Theorem~\ref{main_1} to homological fillings of general $2$-cycles, as the decomposition argument used in the proof requires good topological control. 
Yet, we obtain the following corollary via Gromov's systolic inequality for closed surfaces.

\begin{introcor}\label{Cor:filling_surfaces}
Let $X$ be a proper $\CAT(0)$ space of asymptotic rank at most~$2$. 
For every $\delta > 0$ and every integer $g \ge 0$ there exists a constant $C_g = C_g(\delta)$ such that every closed Lipschitz surface $\Sigma\subset X$ of genus~$g$ extends to a Lipschitz handlebody $H \subset X$ with volume 
\[
\vol(H) \le C_g \cdot\area(\Sigma)^{1+\del}.
\]
\end{introcor}

We would like to point out the formal similarity between our results and another family of filling problems 
in quantitative topology, originally raised by Gromov~\cite{Gro_quant}. 
Specifically, given a null-cobordant closed Riemannian manifold $M$ of bounded geometry 
(say, with sectional curvature $|K| \leq 1$ and injectivity radius at least $1$), 
one wants to minimize the volume of compact Riemannian manifolds $W$ of bounded geometry with $\d W = M$.
An upper bound of order $O(\vol(M)^{1+\delta})$ for every $\delta>0$ was established in~\cite{CDMW_quant}, 
and an analogous result in the PL category has recently been shown in~\cite{MW_pl}.
Notably, as in our case, Gromov originally suggested a linear inequality.

\subsection{Strategy of proof}

We now outline the proof of Theorem~\ref{main_1}. 
It clearly suffices to prove the implications $(1) \Rightarrow (3)$ and $(3) \Rightarrow (2)$.
We show that the first holds in all dimensions, with the sharp constant $\ic_{n+1}$,
for Lipschitz $n$-spheres and asymptotic rank at most $n$ (see Proposition~\ref{Pro:AR}).
The most novel and challenging part is the implication $\rm(3)\Rightarrow\rm(2)$. 
Let us refer to the inequality in~$\rm(2)$ as a {\em $\delta$-isoperimetric inequality}. 
Thus, if the asymptotic rank of a proper $\CAT(0)$ space is at most~$2$,
then we want to show a $\delta$-isoperimetric inequality for all $\delta>0$.

The first ingredient is the case $\delta=1/2$ which corresponds to the Euclidean isoperimetric inequality for $2$-spheres. The result holds for general $\CAT(0)$ spaces, regardless of the asymptotic rank (see Theorem~\ref{Thm:EII}):

\begin{introthm}\label{main_3}
Let $X$ be a $\CAT(0)$ space. Then there exists a constant $C>0$ such that every Lipschitz $2$-sphere $S\subset X$
bounds a Lipschitz $3$-ball $B\subset X$ 
with volume 
\[
\vol(B) \le C \cdot\area(S)^{3/2}.
\]
\end{introthm} 

The argument uses Wenger's approach from~\cite{Wen-Gro} for general integral $n$-cycles and homological fillings,
and the restriction to $2$-spheres allows us to control the topology of the filling.
The proof works more generally for spaces that admit certain {\em Lipschitz coning inequalites}\footnote{There is $C>0$ such that for $n=1,2$, every Lipschitz $n$-sphere $S$ in an $r$-ball bounds a Lipschitz $(n+1)$-ball $B$ with $\vol_{n+1}(B)\leq C r\vol_n(S)$.}. The latter hold, for instance, if $X$ is a $2$-connected Riemannian manifold
with a geometric action of a combable group; see~\cite[Theorem~10.3.6]{Epstein_book}.
However, it remains unclear how to generalize Theorem~\ref{main_3} to higher dimensions.  

Returning to the proof of Theorem~\ref{main_1}, observe that the Euclidean isoperimetric inequality provides the desired implication for all $2$-spheres up to a certain area, depending on $\del$. Furthermore, this threshold can be made arbitrarily large by increasing the constant in the $\del$-isoperimetric inequality. The proof is indirect. Assuming that the conclusion fails for some $\del \in (0,1/2)$, we would like to select a Lipschitz $2$-sphere with minimal area among 
all spheres that violate the $\delta$-isoperimetric inequality for a suitable large constant.
Due to the lack of compactness, we need to work with a sphere $f \colon S^2 \to X$ of nearly minimal area instead. 

These assumptions allow us to prove that $f$ satisfies an {\em intrinsic quadratic isoperimetric inequality} for subdiscs of sufficiently large area. 
We then combine this with an elaborate decomposition process to subdivide the sphere, after a small deformation, into a controlled number of relatively large pieces with effective bounds on their circumference and area. 
This is a crucial point in our argument. As Gromov remarks in~\cite[\S 5.D, (5)(a)]{Gro-AI}, the main difficulty in extending the iterative scheme for the filling of circles
(used, for example, in~\cite[\S 5.A$''_3$]{Gro-AI} or~\cite{Bow_sub}) to higher dimensions is the lack of
nice canonical subdivisions of $n$-dimensional manifolds for $n \ge 2$.

The next step in the proof is to deform $f$ into a Lipschitz sphere $\bar f$ that is {\em piecewise minimal}, whereby we use the $\del$-isoperimetric inequality for spheres of small area to bound the volume of the homotopy. 
Intuitively, one may think of $\bar f$ as being piecewise geodesic on the $1$-skeleton of a triangulation of $S^2$ and sending $2$-simplices to minimal discs. 
After coning off the $0$-skeleton and inserting further minimal discs
into the geodesic triangles thus obtained, it remains to fill a controlled number of {\em minimal tetrahedra}.
This is accomplished by the following linear isoperimetric inequality (see Theorem~\ref{Thm:tetra-filling}). Altogether, we obtain a filling contradicting the assumption on $f$,
which finishes the proof of Theorem~\ref{main_1}.

\begin{introthm}\label{main_4}
Let $X$ be a proper $\CAT(0)$ space of asymptotic rank at most~$2$. 
There exists a constant $\mu>0$ such that each  minimal tetrahedron $\tau\subset X$  satisfies a linear isoperimetric inequality for fillings by balls,
\[\Fillvol(\tau)\leq\mu\cdot\area(\tau).\]
\end{introthm} 

This result represents the key novel idea of this paper.
It crucially exploits properties of minimal discs established in~\cite{Sta}, the aforementioned bound on the filling radius 
from~\cite{GolL}, and an idea of White~\cite{Whi-hom}.
As a by-product of the above strategy, the converse holds 
as well: a linear isoperimetric for minimal tetrahedra implies 
that the aymptotic rank is at most $2$.

%%%%%%%%%%%%%%%%%%%%%%%%%%%%%%%%%%%%%%%%%%%%%%%%%%%%%%%%%%%%%%%%%%%%%%%%%%%%%%%%%

\subsection{Organization}
In Section~\ref{Sec:pre} we introduce our notation and recall some basic facts from metric geometry. We then collect relevant background from the theory of minimal surfaces in metric spaces,  the theory of metric integral currents, and quantitative dimension theory. 
Section~\ref{Sec:ball_fill} contains the proof of Theorem~\ref{main_3}.
In Section~\ref{sec:min_triangle}, we introduce \emph{minimal tetrahedra} -- Lipschitz $2$-spheres composed of four minimal discs filling the four geodesic triangles spanned by a quadruple of points -- and give the proof of Theorem~\ref{main_4}.
Section~\ref{Sec:triangulation} discusses some auxiliary quantitative triangulation results for (pseudo-)metric discs.
Section~\ref{Sec:imp} is devoted to the construction of certain homotopies of spheres that improve their intrinsic structure.
In the final Section~\ref{Sec:proof}, we complete the proof of Theorem~\ref{main_1} and derive Corollary~\ref{Cor:filling_surfaces}.

\subsection{Acknowledgements}
The first author is grateful to the Max Planck Institute for Mathematics in Bonn for its hospitality during several research visits, while part of the work at this paper was carried out.
The second author was partially supported by Swiss National Science Foundation Grant~197090.
The last author was partially supported by DFG grant SPP 2026.

%%%%%%%%%%%%%%%%%%%%%%%%%%%%%%%%%%%%%%%%%%%%%%%%%%%%%%%%%%%%%%%%%%%%%%%%%%%%%%% 

\section{Preliminaries}\label{Sec:pre}

\subsection{Basics}

In the Euclidean space $\R^n$, we denote the open unit ball by $U^n$, the closed unit ball by $B^n$, and its boundary, the unit sphere, by $S^{n-1}$. The Lebesgue measure of $B^n$ is denoted by $\om_n$.
A {\em disc} $D$ is a topological space homeomorphic to $B^2$.

We denote the distance in a metric space $X$ by $d$. 
The closed ball in $X$ of radius $r$ and center $x\in X$ is denoted by
\[B(x,r)= \{y\in X: d(x, y)\leq r\}. \]
If closed balls in $X$ are compact, then $X$ is called {\em proper}.
For compact subsets $A,A'\subset X$ we denote by $|A,A'|_H$ their Hausdorff distance.

For $n\geq 0$, the $n$-dimensional Hausdorff measure on $X$ is denoted by $\cH^n$. 
The normalizing constant is chosen so that on the Euclidean space $\R^n$ the Hausdorff measure $\cH^n$ equals the Lebesgue measure.

A continuous curve in $X$ is called {\em rectifiable} if it has finite length. An {\em arc} is a homeomorphism from a compact interval to a subset of $X$. A \emph{geodesic} is an isometric map  from an interval into $X$. We refer to the images of the corresponding maps by the same names. 
A {\em Jordan curve} $\Ga$ in a metric space $X$ is a closed subset homeomorphic to $S^1$.

The space $X$ is called \emph{a geodesic space} if any pair of points in $X$ is connected by a geodesic.
A \emph{$\CAT(0)$ space} $X$ is a complete geodesic space such that distances on  a geodesic triangle in $X$
are bounded from above by the corresponding distances on its comparison triangle in $\R^2$.
We assume some familiarity with $\CAT(0)$ geometry and refer the reader to \cite{BGS, ballmannbook, BriH} for more information.

\subsection{Lipschitz maps, area and volume}

Let $U \sub \R^n$ be an open set, and let $X$ be a metric space.
Every Lipschitz map $f\colon U\to X$ has a {\em metric differential} at almost every point $x\in U$ in the following sense. 
There exists a unique semi-norm on $\R^n$, denoted $|df_x|$, such that 
\begin{equation*}
\lim_{x'\to x}\frac{d(f(x'), f(x)) - |df_x|(x'-x)}{\|x'-x\|} = 0.
\end{equation*}
The map $x\mapsto |df_x|$ into the space of semi-norms has a Borel measurable representative. 
The \emph{Jacobian $\Jac(s)$} of a semi-norm $s$ on $\R^n$ is the Hausdorff $n$-measure in $(\R^n, s)$ of the Euclidean unit cube if $s$ is a norm, and $\Jac(s)=0$ otherwise. 
The {\em volume} of a Lipschitz map $f\colon U\to X$ is given by the {\em area formula}
\[
\vol_{n}(f) := \int_U \Jac(|df_x|)\,dx = \int_X \#(f^{-1}\{y\}) \,d\cH^n(y).
\] 
See~\cite{Kirch, KS, LWPlateau}. If $X$ is a $\CAT(0)$ space, then for almost 
every $x \in U$ where the metric differential is non-degenerate, the norm $|df_x|$
is induced by an inner product on~$\R^n$ (see \cite[Lemma~4.3]{Wen-FI}).
If $n = 1$ or $n = 2$, we will speak of {\em length} or {\em area} instead of
volume and write $\length(f)$ or $\area(f)$, respectively.
The volume of a Lipschitz map from a relatively open subset $U$ of $S^n$ is defined similarly.

A Lipschitz map from an $n$-sphere or $(n+1)$-ball to $X$ will be called a {\em Lipschitz $n$-sphere} or {\em Lipschitz $(n+1)$-ball}, respectively.
For convenience, we will sometimes also refer to their images in $X$ as Lipschitz $n$-spheres or Lipschitz $(n+1)$-balls, often denoted by $S\subset X$ and $B\subset X$, respectively.  
The volumes $\vol_n(S)$ and $\vol_{n+1}(B)$ are then understood to be the volumes of 
the corresponding maps or, equivalently, as the Hausdorff measures counted with multiplicities. If the dimension is clear from the context, we omit the subscript and simply write $\vol(S)$ and $\vol(B)$. 

We say that a Lipschitz ball $\bar f\colon B^{n+1}\to X$ {\em fills} the Lipschitz sphere $f:=\bar f|_{S^n}$, or that $S$ {\em bounds}  $B$ where $S=f(S^n)$ and $B=\bar f(B^{n+1})$.
For a given Lipschitz sphere $f\colon S^n\to X$, we define its {\em filling volume} by
\[
\Fillvol(f):=\inf \{\vol(\bar f)\mid \bar f\colon B^{n+1}\to X \mbox{ a Lipschitz ball filling }f \}.
\]
If no such $\bar f$ exists, then $\Fillvol(f) = \infty$.
\medskip

In order to build Lipschitz balls of small volume that fill a given Lipschitz sphere,
it is useful to decompose the sphere. To do this, we need some control on the fibers of Lipschitz functions.
The following result, which will be used repeatedly, is a slight variation of
\cite[Theorem~2.5]{AlBiCr_Sard}.

\bprop\label{Pro:reg_fiber}
Let $f\colon B^2\to \R$ be a non-constant Lipschitz function. We denote its
fibers $f^{-1}(t)$ by $\Pi_t$. Then for almost every value $t\in\R$ the following holds.
\begin{enumerate}
	\item $\Hm^1(\Pi_t)$ is finite;
 \item The differential of $f$ is non-trivial at $\Hm^1$-almost every point in $\Pi_t$;
	\item $\Pi_t$ decomposes as a disjoint union of rectifiable Jordan curves,  rectifiable
 arcs with boundary points on $S^1$, and an $\Hm^1$-negligible
	set.
\end{enumerate}
\eprop

\begin{proof}%[Sketch of Proof]
Extend $f$ to a Lipschitz function $F \colon \R^2 \to \R$ with compact support.
Now~\cite[Theorem~2.5]{AlBiCr_Sard} shows that, for almost every $t$,
the fiber $F^{-1}(t)$ has finite $\Hm^1$ measure, every connected component $\Ga$
of $F^{-1}(t)$ with $\Hm^1(\Gam) > 0$ is a Jordan curve,
and the union of the remaining connected components has $\Hm^1$ measure zero.
Furthermore, every such Jordan curve $\Ga$ admits a unit speed parametrization 
$s \mapsto \gam(s)$ such that $(\nabla F(\gam(s)),\dot\gam(s))$ 
is a positively oriented basis of $\R^2$ for almost every $s$.
By the coarea formula, $t$ can be chosen in addition such that $F^{-1}(t) \cap S^1$ is a finite set.
Then every Jordan curve $\Gam \sub F^{-1}(t)$ is either contained in $B^2$ or intersects $B^2$ 
in finitely many arcs or points.
\end{proof}

The above statement holds true in similar form for maps defined on $S^2$. In this case,
the fibers generically decompose into a disjoint union of Jordan curves and an 
$\Hm^1$-negligible set, no arcs appear.

\subsection{Plateau problem and minimal discs} \label{Sec:minimal}
Recall that $U^2$ denotes the open unit disc in $\R^2$.
Let $X$ be a $\CAT(0)$ space and $\Ga\subset X$ a closed rectifiable curve.
The Plateau problem for $\Ga$ asks to find a disc of minimal area spanning $\Ga$.
For Jordan curves it was solved in every complete metric space by Lytchak-Wenger~\cite{LWPlateau} and the result 
was extended to general curves by Creutz~\cite{Cr_sing}.
For a Lipschitz parametrization $\ga \colon S^1 \to \Gam$, 
consider the family of competitors
\[
\Lambda(\ga,X):=\{u\in W^{1,2}(U^2,X)|\ u \text{ spans } \ga\},
\]
where $W^{1,2}(U^2,X)$ denotes the Sobolev space of finite energy maps, and an element $u$ is said to {\emph{span}} $\ga$ if its trace $\tr(u)\in L^2(S^1,\Ga)$ has a continuous representative homotopic to $\ga$. We will not recall the relevant definitions from Sobolev theory with values in metric spaces
here, but refer the reader to \cite{LWPlateau} for details. For our purposes, it is enough to know that every finite energy map has a well-defined area, and every Lipschitz map $v\colon B^2\to X$
restricts to an element in $W^{1,2}(U^2,X)$; its trace is simply $v|_{S^1}$ and its area is $\area(v)$ as defined above. 
Plateau's problem then asks to find an area minimizer $v\in\Lambda(\ga,X)$, meaning
\[\area(v)=\inf\{\area(u)|\ u\in\Lambda(\ga,X)\}.\]
%A solution to the Plateau problem is  called a {\em minimal disc}.
The following solution is well-known to experts and can be assembled from existing results in the literature.
The key point is that it yields a Lipschitz map on the closed disc,
whereas the (quasi-conformal) solution to the Plateau problem from~\cite{LWPlateau, Cr_sing} is only locally Lipschitz in the interior.

Recall that a {\em disc retract} is a Peano continuum homeomorphic to a non-separating subset of the plane.
A {\em $\CAT(0)$ disc retract} is a $\CAT(0)$ space homeomorphic to a disc retract.

\begin{Thm}\label{Thm:Plateau}
Let $X$ be a proper $\CAT(0)$ space and $\Ga\subset X$ a closed rectifiable curve.
Then there exists a Lipschitz disc $v\colon B^2\to X$ that solves the Plateau problem
for $\Ga$. Moreover, $v$ satisfies the following properties.
\begin{itemize}
\item (monotonicity of density ratios)
\[
\pi \le \frac{\area(v(B^2) \cap B(x,r))}{r^2} \le \frac{\area(v(B^2) \cap B(x,s))}{s^2}
\]
for $x \in v(B^2)$ and $0 < r \le s < d(x,v(S^1))$. 

\item $v$ admits a factorization as $v=v'\circ v''$ with $v''\colon B^2\to Z$ and $v'\colon Z\to X$, where the space $Z$ is a $\CAT(0)$ disc retract and $v''$ is a surjective Lipschitz map with connected fibers. Furthermore, $v'$ is area minimizing among Lipschitz maps from $Z$ and preserves the length of every rectifiable curve.
\end{itemize}
\end{Thm}

Recall that the intrinsic metric space $Z_u$ associated to a map $u\colon B^2\to X$
is the quotient metric space obtained from the pseudo-metric $d_u$  on $B^2$ defined by
\[d_u(z_1,z_2)=\inf\{\length(u\circ\gamma)\mid \gamma \mbox{ a path in }B^2\mbox{ joining }z_1\mbox{ and }z_2\}.\]

\proof
By \cite[Theorem~4.3]{Cr_sing} there exists a solution to the Plateau problem
$u\in\Lambda(\ga,X)$ that is locally $\frac{1}{3}$-H\"older continuous and extends to a 
$\frac{1}{27}$-H\"older continuous map on $B^2$ which we will still denote by $u$.
Since $u$ is harmonic, it is locally Lipschitz in $U^2$ by \cite{KS}. Moreover, the associated intrinsic metric space $Z_u$ is $\CAT(0)$,
the natural projection $\pi_u\colon B^2\to Z_u$ has connected fibers, and $Z_u$ is a disc retract 
 \cite{LyWa_min,PS}. By \cite[Corollary~3.2]{LWint}, we have $\pi_u\in W^{1,2}(U^2,Z_u)$
 and the approximate metric differentials of $u$ and $\pi_u$ coincide almost everywhere.
 In particular, $\area_u(W)=\area_{\pi_u}(W)$ for every Borel set $W\subset U^2$.
Let $c\colon S^1\to\d Z_u$ be a Lipschitz parametrization representing $\tr(\pi_u)$.
The proof of \cite[Lemma~30]{Sta} provides a surjective Lipschitz
map $\ph\colon B^2 \to Z_u$ with connected fibers that extends $c$.
Then
\[
\area(\ph|_W)=\Hm^2(\ph(W))
\] 
for every Borel set $W\subset B^2$ by the aforementioned area formula. %~\cite[Theorem~7]{Kirch}. 
Set $v=\bar u\circ \ph$, where $\bar u \colon Z_u \to X$ is the unique $1$-Lipschitz map such that $u = \bar u \circ \pi_u$. Then, for any ball $B(x,r)\subset X$, 
\begin{align*}
&\area(u(B^2)\cap B(x,r))=\area(\bar u(Z_u)\cap B(x,r))=\Hm^2(\bar u^{-1}(B(x,r)))\\
&\quad=\area_\ph(v^{-1}(B(x,r)))=\area_v(v^{-1}(B(x,r)))=\area(v(B^2)\cap B(x,r)).
\end{align*}
In particular, $v$ satisfies the required monotonicity formula since $u$ does \cite[Corollary~63]{Sta}.
\qed

\medskip

We call a Lipschitz disc as above a {\em minimal disc}.

\subsection{Disc decompositions}\label{Sec:gen_discs}

Let $\Sigma\subset\R^2$ be a compact connected planar surface with $k+1\geq 1$ boundary components. Note that $\R^2\setminus \Sigma$ has precisely one unbounded component and $\Sigma$ has one {\em outer boundary circle} $\sig_0$ and $k$
{\em inner boundary circles} $\sig_1,\ldots,\sig_k$. If $k$ is positive, then we call such a surface a {\em disc with $k$ holes}. We call it a {\em polygonal disc with $k$ holes} if in addition every boundary circle is a piecewise geodesic Jordan curve. 

The following provides a weak surrogate for a triangulation. 
Given a disc with holes $\Sig$, we define a {\em disc decomposition} of $\Sig$ by means of a connected finite Lipschitz graph $\Ga\subset \Sig$ with $\d \Sig\subset \Ga$ such that each vertex has valence at least three, adjacent vertices are connected by a unique edge, and the closure of each component of $\Sig\setminus \Ga$ is a disc.
In particular, we have a covering by discs $D_i$:
\[
  \Sig = D_1 \cup \ldots \cup D_n.
\]
We call a disc decomposition {\em polygonal} if the edges of $\Ga$ are piecewise geodesic.
Similarly, we define (polygonal) disc decompositions of $2$-spheres.

Recall that a map between topological spaces is called {\em monotone}
if it has connected fibers.

\begin{Lem}\label{Lem:mon_ext}
Let $D \sub \R^2$ be a closed disc with rectifiable boundary. Then every Lipschitz parametrization $\gam \colon
S^1 \to \d D$ extends to a surjective monotone Lipschitz map 
$\ph \colon B^2 \to D$.
\end{Lem}

\begin{proof}
Denote by $Y$ the interior of $D$ equipped with the induced path metric.
Since $\d D$ is rectifiable, any pair of points in $Y$ can be joined by a polygon of length at most half the length of $\d D$. Thus,  the completion $\bar Y$ is compact and therefore a
$\CAT(0)$ disc \cite[Proposition~12.1, Lemma~12.2]{LWcurv}.
Denote by $f\colon\bar Y\to D$ the natural surjective 1-Lipschitz map.
Since  $\bar Y\setminus\d \bar Y$ is a length space \cite[Theorem~1.3]{LWcurv}, and $f$ is a local isometry in the interior, $f$  preserves the length of every rectifiable curve in $\bar Y$. 
Since $f$ restricts to a homeomorphism between the interiors, $f$ must be a homeomorphism.
Thus, $f^{-1}\circ\ga$ is a Lipschitz parametrization of $\d\bar Y$. By \cite[Lemma~30]{Sta}, 
$f^{-1}\circ\ga$ extends to a monotone surjective Lipschitz map
$\psi\colon B^2\to\bar Y$.
Thus, the composition $\ph=f\circ\psi$
has the required properties.
\end{proof}

The following proposition will often be used to reparametrize Lipschitz spheres.

\begin{Pro} \label{Pro:pre-homotopy}
Let $\Ga\subset S^2$ be a finite Lipschitz graph.
Then, for every $\eps>0$, there exists a graph $P\subset S^2$
and a monotone Lipschitz map $F\colon S^2\to S^2$
with the following properties.
\begin{enumerate}
\item The edges of $P$ are piecewise geodesics;
\item $P$ is $\eps$-Hausdorff close to $\Ga$;
    \item $F$ restricts to a homeomorphism $P\to \Ga$;
    %\item $\area(F|_M)=\area(F(M))$ for every Borel set $M\subset S^2$;
    \item $F$ is uniformly $\eps$-close to the identity.
\end{enumerate}
\end{Pro}

\begin{proof}
After possibly inserting finitely many Lipschitz arcs, we may assume that $\Ga$ is connected and  the closure of
every component of $S^2\setminus \Ga$ is a closed disc of diameter at most $\frac{\eps}{3}$.
To define $P$, start with $\Ga$ and replace every edge
$e$ by a $\delta$-close piecewise geodesic $e'$, for some $\delta\ll\eps$, such that the new edges only intersect in the original vertices. Then the boundary of every component $V$ of $S^2\setminus P$
is a Jordan curve consisting of finitely many geodesic segments. 
By \cite{Tuk_schoen}, the closure $\bar V$ is a bilipschitz disc. Its diameter is at most $\frac{\eps}{3}+2\delta$.

To define $F$, we first send every vertex of $\Ga$ to itself. On the edges of $P$, we choose for every edge $e'$ a bilipschitz map to the closed unit interval and compose it with a constant speed parametrization of the corresponding edge $e$ of $\Ga$. Now we use Lemma~\ref{Lem:mon_ext} to find monotone Lipschitz extensions of $F$ to the components of $S^2\setminus P$. 
By construction, $F$ is monotone Lipschitz and  sends $P$ 
homeomorphically to $\Ga$. Moreover, it maps every component
of $S^2\setminus P$ to the corresponding component of $S^2\setminus \Ga$.
Therefore $d(F(x),x)<\eps$ for all points $x\in S^2$. 
\end{proof}

For $\eps\in(0,1)$ define $\rho_\eps\colon B^2\to B^2$ by 
\[ \rho_\eps(x) =
  \begin{cases}
    \frac{1}{1-\eps}\cdot x      & \quad \text{if } \|x\|\leq 1-\eps, \\
    \frac{x}{\|x\|}  & \quad \text{else}.
  \end{cases}
\]
Let $D$ be a bilipschitz disc. Then we call a map $\psi_\eps\colon D\to D$
an {\em $\eps$-squeeze} if $\psi_\eps=\beta\circ\rho_\eps\circ\beta^{-1}$
for some bilipschitz map $\beta\colon B^2\to D$.
Note that in this case, if  $f\colon D\to X$
a Lipschitz map to a metric space $X$, then the pseudo-metric spaces $(D,d_f)$ and $(D,d_{f\circ\psi_\eps})$ are naturally isometric.
Moreover, if
$\ga$ is a path in $D$ with $\d\ga\subset\d D$, 
then there exists a path $\tilde\ga$
in $D$ with $\d\tilde\ga\subset\d D$ 
and such that $\psi_\eps$ maps $\tilde\ga$ monotonically onto $\ga$.

\begin{Lem}\label{Lem:2nd_path}
Let $\Sigma\subset S^2$ be a disc with $k\geq 1$ holes and rectifiable boundary. Let $f\colon\Sigma\to X$ be a Lipschitz map. Let $s>0$ be
such that $d_f(x,\sig_0)< s$ for every $x\in \sig_1\cup\ldots\cup\sig_k$. 
Then, for every $\eps>0$ there exist a disc with $k$ holes $\Sig'\subset S^2$ and
piecewise geodesic boundary components $\sig_0',\ldots,\sig_k'$,  a surjective monotone Lipschitz map $\psi\colon S^2\to S^2$ that sends $\Sig'$ onto $\Sig$, and $k$ pairs of piecewise geodesic paths 
$\beta_j^\pm\subset\Sigma'$ with the following properties.

\begin{enumerate}
    \item $|\Sig',\Sig|_H<\eps$ and $\psi$ is $\eps$-close to the identity;
    \item $\length_{f\circ \psi}(\sig'_i)=\length_f(\sig_i)$, $i=0,\ldots,k$;
    \item $d_{f\circ\psi}(x',\sig'_0)< s$ for every $x'\in \sig'_1\cup\ldots\cup\sig'_k$;
    \item $\length_{f\circ\psi}(\beta_j^\pm)<s$, $j=1,\ldots,k$;
    \item $|\beta_j^-,\beta_j^+|_H<\eps$ and $\beta_j^-\cap\beta_j^+=\emptyset$, $j=1,\ldots,k$;
    \item $\beta_j^\pm$ joins $\sig'_j$ to $\sig'_0\cup\Lambda_{j-1}$ and $\beta_j^\pm\cap(\sig'_0\cup\Lambda_{j-1}\cup\sig'_j)=\d\beta_j^\pm$
    where $\Lambda_l=\bigcup_{i=1}^{l}(\beta_i^-\cup\beta_i^+)$ for $1\leq l\leq k$.
\end{enumerate}
In particular, the $1$-rectifiable set $\Ga'=\d\Sig'\cup \Lambda_k$ carries a natural graph structure and induces a polygonal disc decomposition of $\Sig'$.
\end{Lem}

\begin{figure}[ht]
    \centering
    \includegraphics[scale=0.35,trim={1cm 4cm 0cm 0cm},clip]{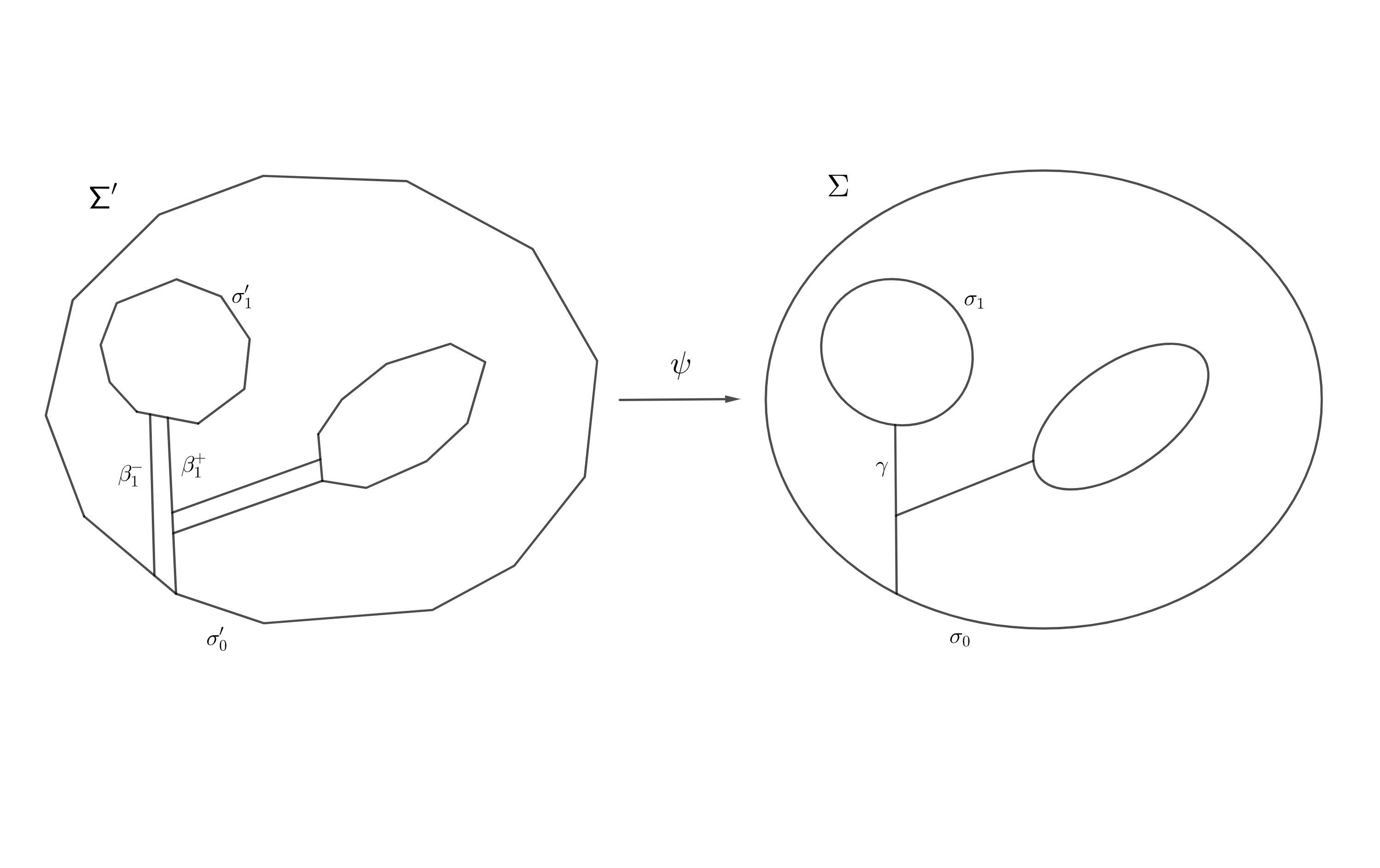}
    \caption{A polygonal resolution of a disc with two holes.}
    \label{fig:poly_resol}
\end{figure}

\begin{proof}
We prove the statement by induction on $k$.
For $k=1$, choose an arc $\ga\subset\Sig$ joining $\sig_1$
to $\sig_0$ with $\ga\cap\d \Sig=\d\ga$ and $\length_f(\ga)<s$.
For $\eps>0$, apply Proposition~\ref{Pro:pre-homotopy}  to the finite Lipschitz graph  
$\Ga_1=\d\Sig\cup\ga$ to obtain an $\eps$-close piecewise geodesic graph $P_1$ and a map $F\colon S^2\to S^2$ $\eps$-close to the identity. We define $\beta^-$ to be the edge of $P_1$ corresponding to $\ga$, and $\Sig'$ to be the annulus with boundary $P_1\setminus \beta^-$. We then add an auxiliary piecewise geodesic $\al\subset \Sig'$ to $P_1$ such that the closures of the components of $\Sig\setminus (P_1\cup\al)$ are bilipschitz discs. 
Next, we choose an $\eps$-squeeze on each of those discs to define a monotone Lipschitz map
$\psi'\colon \Sig'\to\Sig'$. We then find a piecewise geodesic 
$\beta^+$, disjoint, but $\eps$-Hausdorff close to $\beta^-$,
that is mapped monotonically onto an arc in $N_\eps(\beta^-)\cap P_1$ by $\psi'$.
We obtain the induction base by setting $\psi=F\circ\psi'$.
The induction step is similar. We ignore the boundary circle $\sig_k$ and apply the induction hypothesis to the disc with $(k-1)$ holes. So we may assume that $\Sig$ has piecewise geodesic boundary and
we have already found $k-1$ pairs of piecewise geodesic paths 
$\beta_j^\pm\subset\Sigma$ with the right properties.
To proceed, we choose an arc $\ga$ with $\length_f(\ga)<s$ joining $\sig_k$
to $\sig_0\cup\Lambda_{k-1}$ where $\Lambda_l=\bigcup_{i=1}^{l}(\beta_i^-\cup\beta_i^+)$ for $1\leq l\leq k$.
As above, we apply Proposition~\ref{Pro:pre-homotopy}  to replace $\ga$
by a piecewise geodesic $\beta_k^-$ and then use $\eps$-squeezes to
produce $\beta_k^+$.
\end{proof}

\subsection{Nagata dimension} \label{Sect:nag-dim}

We recall the definition of the Nagata dimension (see~\cite{LanS}). A family
$\cB = \{B_i\}_{i \in I}$ of subsets of a metric space $X$ is called
{\em $D$-bounded} if every $B_i$ has diameter at most $D$. We say that $\cB$
has {\em $s$-multiplicity at most $m$} if every set of diameter $\le s$
in $X$ meets no more than $m$ members of the family. The {\em Nagata dimension} 
of $X$ is {\em at most} $n$ if and only if there exists a constant $c$ such that 
for every $s > 0$, $X$ admits a $cs$-bounded covering with $s$-multiplicity
at most $n + 1$.

If $\{B_i\}_{i \in I}$ is a $D$-bounded covering of a subset $S \sub X$ with 
$s$-multiplicity at most $n+1$, and if $\eps \in (0,s/2)$, then the 
corresponding family of the open $\eps$-neighborhoods $U_\eps(B_i)$ forms a
$(D + 2\eps)$-bounded cover of $U_\eps(S)$ with $(s - 2\eps)$-multiplicity at most $n+1$.
In particular, with a slight adjustment of the constant $c$, the coverings
in the above definition can be taken open.

The Nagata dimension will be used through the following lemma.

\begin{Lem} \label{Lem:complex}
Suppose that $X$ is a proper $\CAT(0)$ space, and $K \sub X$ is a compact subset
with a $cs$-bounded covering with $s$-multiplicity at most $n + 1$, for some integer $n$ and positive constants $c$ and $s$. Then there exist
\begin{enumerate}
\item[\rm (1)]
a finite simplicial complex $\Sig$ of dimension at most $n$, metrized as a subcomplex of some simplex of edge length $s$ in a Euclidean space;
\item[\rm (2)]
a constant $L > 0$ depending only on $n$ and $c$, and 
an $L$-Lipschitz map $\psi \colon K \to \Sig$;
\item[\rm (3)]  
a Lipschitz map $\phi \colon \Sig \to X$ with $\phi(\Sig^{(0)}) \sub K$ that is
$L$-Lipschitz on every simplex of\/ $\Sig$, such that $\phi \circ \psi$ is $L$-Lipschitz and satisfies $d(x,\phi \circ \psi(x)) \le Ls$ for all $x \in K$.
\end{enumerate}
\end{Lem}

\begin{proof}
By the remark preceding the lemma, we may assume that the given covering is finite, as $K$ is compact. Now one can proceed as in~\cite[Proposition~6.1]{BasWY}, using a subordinate Lipschitz partition of unity. 
The fact that $\phi$ is Lipschitz, with a constant that is possibly not bounded in terms of $n$ and $c$, is not stated there; here, it holds since $\phi$ is locally Lipschitz and $\Sig$ is compact.
\end{proof}

We will also use the following related concept.
The {\em Assouad dimension} of a metric space $X$ is defined as the infimum of all (real) numbers 
$\beta > 0$ with the property that there exists some $C > 1$ 
such that, for every $\eps \in (0,1)$, every set of diameter $D$ in $X$ can 
be covered by no more than $C \eps^{-\beta}$ sets of diameter at most $\eps D$
(see~\cite[Definition~10.15]{Heinonen_book}). Unlike the Nagata dimension, this is in general not an integer. 
However, by~\cite[Theorem~1.1]{LeDR}, the Nagata dimension (as well as the Hausdorff dimension) is at most the Assouad dimension.
The result is quantitative, in that the constant $c$ in the definition of 
the Nagata dimension depends only on the Assouad dimension and the respective constant $C$.

\subsection{Metric integral currents}\label{Sec:currents}

For a proper metric space $X$, we let $\bI_{*,\cs}(X)$ denote the chain complex 
of metric integral currents with compact support in $X$, see \cite{AmbK,Lan_LC}. 
As the use of currents will be limited to the proofs of Theorem~\ref{Thm:tetra-filling} and 
Proposition~\ref{Pro:AR}, we give only a brief outline and refer 
to~\cite{GolL} for a more detailed introduction.

Formally, a current $T \in \bI_{n,\cs}(X)$ is an $(n+1)$-linear functional on the space 
of $(n+1)$-tuples of real-valued Lipschitz functions on $X$, subject to some further conditions. 
As a basic example, every compact set $B \sub \R^n$ with finite perimeter, in particular 
every closed ball or polytope, induces a current $T = \bb{B} \in \bI_{n,\cs}(\R^n)$ defined by
\[
\bb{B}(f_0,f_1,\ldots,f_n) := \int_B f_0 \det[\partial_j f_i]_{i,j=1}^n \,dx,
\]
where the partial derivatives exist almost everywhere by Rademacher's theorem.
Notice that for smooth functions $f_i$, this corresponds to the classical (de Rham) current 
given by the integration of the differential $n$-form $f_0 \,df_1 \wedge \ldots \wedge df_n$ on $B$.

For a current $T \in \bI_{n,\cs}(X)$, the boundary $\d T \in \bI_{n-1,\cs}(X)$ (if $n \ge 1$)
and the push-forward $\ph_\#T \in \bI_{n,\cs}(Y)$ under a Lipschitz map $\ph \colon X \to Y$
into another proper metric space $Y$ are defined by the relations
\begin{align*}
\d T(f_0,\ldots,f_{n-1}) &:= T(1,f_0,\ldots,f_{n-1}), \\
\ph_\#T(g_0,\ldots,g_n) &:= T(g_0 \circ \ph,\ldots,g_n \circ \ph).
\end{align*}
Note that $\d(\ph_\#T) = \ph_\#(\d T)$. Associated with $T \in \bI_{n,\cs}(X)$ is
a compactly supported Borel measure $\|T\|$ on $X$ with finite mass 
$\M(T) := \|T\|(X)$, and the support $\spt(T)$ of $T$ is defined as the support 
of $\|T\|$. If $\ph \colon X \to Y$ is a $\lam$-Lipschitz map, then 
$\M(\ph_\#T) \le \lam^n\,\M(T)$.
In the case of the above example $T = \bb{B}$, the measure 
$\|T\|$ is just the restriction of the Lebesgue measure to $B$. 
If $X$ is a $\CAT(0)$ space and $u \colon B^2 \to X$ is a Lipschitz disc with
connected fibers, then $\bigl\| u_\#\bb{B^2} \bigr\|(B) = \area(u(B^2) \cap B)$ for every
Borel set $B \sub X$.

Every singular Lipschitz $n$-chain $\sum_{i=1}^N a_i \sig_i$ in $X$, where $a_i \in \Z$ and 
the maps $\sig_i \colon \Del^n \to X$ are Lipschitz on an $n$-simplex $\Del^n \sub \R^n$,
induces a corresponding current $\sum_{i=1}^N a_i \,\sig_{i\#}\bb{\Del^n} \in \bI_{n,\cs}(X)$.
If $X$ is locally Lipschitz contractible in a suitable sense, then the resulting chain map from 
singular Lipschitz chains to compactly supported integral currents induces an isomorphism of 
the respective homology theories (see~\cite{RiSch_hom} and \cite{Mit_coin}). 
Moreover, with some additional assumptions on $X$, the subspace induced by
Lipschitz $n$-chains is dense in $\bI_{n,\cs}(X)$ with respect to the 
metric $(T,T') \mapsto \bN(T-T')$, where $\bN(T) := \M(T) + \M(\d T)$
(see~\cite[Theorem~1.3]{Gol} and~\cite[Corollary~1.5]{BasWY}).
One of the main benefits from passing to $\bI_{n,\cs}(X)$
is the following compactness theorem (see~\cite[Theorem~5.2 and Theorem~8.5]{AmbK}).

\begin{Thm} \label{Thm:cptness}
Given a sequence of integral currents $T_k \in \bI_{n,\cs}(X)$ with supports
in a fixed compact set $K \sub X$ and $\sup_k \bN(T_k) < \infty$, there exists a 
subsequence converging weakly, that is, pointwise as functionals, to an integral 
current $T \in \bI_{n,\cs}(X)$ with support in $K$.
\end{Thm}

The mass is lower semi-continuous with respect to weak convergence: 
$T_k \to T$ implies that $\M(T) \le \liminf_{k \to \infty} \M(T_k)$.
For a cycle $S \in \bI_{n,\cs}(X)$ of dimension $n \ge 1$, 
a current $V \in \bI_{n+1,\cs}(X)$ is a {\em filling} of $S$ if $\d V = S$.
If $X$ is a $\CAT(0)$ space, then a geodesic cone construction shows that such 
a $V$ exists (see~\cite[Theorem~4.1]{Wen-FI}). A filling is {\em minimizing} if 
it has minimal mass among all fillings of its boundary cycle.
The first part of the following analogue of Theorem~\ref{Thm:Plateau} 
now follows directly from Theorem~\ref{Thm:cptness} together with the fact 
that the nearest point retraction $X \to K$ is $1$-Lipschitz.
For the assertion about the density, see~\cite[Corollary~4.4 and~(4.28)]{Wen-FI}.

\begin{Thm} \label{Thm:plateau-n}
Let $X$ be a proper $\CAT(0)$ space, and let $S \in \bI_{n,\cs}(X)$ be a cycle 
with support in a compact ball $K$. Then there exists a minimizing filling
$V \in \bI_{n,\cs}(X)$ of\/ $S$ with support in $K$. Furthermore, every minimizing 
filling $V$ satisfies
\[
\om_{n+1} \le \frac{\|V\|(B(x,r))}{r^{n+1}} \le \frac{\|V\|(B(x,s))}{s^{n+1}}
\]
whenever $x \in \spt(V)$ and $0 < r \le s < d(x,\spt(S))$.
\end{Thm}

See also~\cite[Theorem~1.6]{Wen-EII}, where $X$ is not assumed to be proper.

Towards a linear isoperimetric inequality in dimensions above the asymptotic rank, 
the following result for a restricted class of cycles was proved (in a more general 
form) in~\cite[Theorem~1.1 and Theorem~7.3]{GolL}.
We say that $S \in \bI_{n,\cs}(X)$ has {\em $c$-controlled density}, for some 
constant $c > 0$, if $\|S\|(B(x,r))/r^n \le c$ for all $x \in X$ and $r > 0$.

\begin{Thm} \label{Thm:fillrad}
Let $X$ be a proper $\CAT(0)$ space of asymptotic rank at most~$n \ge 1$.   
Then for all $c > 0$ there exists $\rho = \rho(X,n,c) > 0$ such that every 
cycle $S \in \bI_{n,\cs}(X)$ with $c$-controlled density has a filling 
$V \in \bI_{n+1,\cs}(X)$ such that\/ $\M(V) \le \rho\,\M(S)$ and the 
support of\/ $V$ is within distance at most~$\rho$ from $\spt(S)$.
\end{Thm}

The bound on the filling radius will be used in the proof of Theorem~\ref{Thm:tetra-filling}.

%%%%%%%%%%%%%%%%%%%%%%%%%%%%%%%%%%%%%%%%%%%%%%%%%%%%%%%%%%%%%%%%%%%%%%%%%%%%%%%

\section{The Euclidean isoperimetric inequality for spheres}\label{Sec:ball_fill}

We first recall Reshetnyak's majorization theorem~\cite{Res} which implies the sharp Euclidean isoperimetric inequality
for curves.

\begin{Pro} \label{Pro:reshetnyak}
Let $X$ be a $\CAT(0)$ space. Then for every closed unit speed curve 
$\alpha \colon [0,l] \to X$ there exist a compact convex set $C \sub \R^2$, 
a surjective closed unit speed curve $\beta \colon [0,l] \to \d C$ that is 
simple except in the degenerate case when $C$ is a segment of length $l/2$,
and a $1$-Lipschitz map $f \colon C \to X$ such that 
$f(\beta(t)) = \alpha(t)$ for all $t \in [0,l]$.
In particular,
\[
\area(f) \le \area(C) \le \frac{1}{4\pi} \cdot l^2.
\]
\end{Pro}

We now turn to the Euclidean isoperimetric inequality for $2$-spheres, Theorem~\ref{main_3},
which we restate as follows.

\begin{Thm}\label{Thm:EII}
Let $X$ be a $\CAT(0)$ space. Then there exists a constant $C>0$ such that every Lipschitz $2$-sphere
$f \colon S^2 \to X$ extends to a Lipschitz ball $\bar f \colon B^3 \to X$ 
with volume 
\[
\vol(\bar f) \le C \cdot\area(f)^{3/2}.
\]
\end{Thm}

The following argument is similar to Wenger's proof of Gromov's filling inequality~\cite{Wen-Gro}.

\proof
Let $f\colon S^2\to X$ be a Lipschitz $2$-sphere with image $S\subset X$.
We will decompose $f$ into finitely many Lipschitz $2$-spheres $S=\bigcup_{a\in\mathcal A} S_a$ such that each $S_a$ can be efficiently filled by coning.
Our decomposition process works as follows. If $\Ga\subset S^2$ is a bilipschitz Jordan curve, 
then we fill $f(\Ga)$ by a Lipschitz disc $D\subset X$ using the isoperimetric inequality, thereby obtaining two Lipschitz $2$-spheres
$S^\pm\subset X$ that intersect in $D$.
If we then fill both of them by Lipschitz $3$-balls, we obtain a corresponding ball-filling for $S$.
In the following, $\length$, $\area$ and $\vol$ of the image $Z$ of a Lipschitz map will denote the appropriate Hausdorff 
measure of $Z$, counted with multiplicities. For instance, if $f\colon S^2\to X$ is a Lipschitz $2$-sphere
with image $S$ and the map $f$ is clear from the context, then we simply write
\[\area(S)=\int_X \# f^{-1}(x)\, d\Hm^2(x).\]
Note that, by the area formula, this coincides with $\area(f)$.

\bslem
There exist $q\in(0,1)$ and $C>0$,  such that we can fill finitely many disjoint Jordan curves $\Ga_j\subset S$ by Lipschitz discs $D_j$ to produce a decomposition \[S\cup\bigcup_j D_j=\bigcup_{a\in\mathcal A} S_a\cup\bigcup_{b\in\mathcal B} S_b\]
into Lipschitz $2$-spheres with the following properties.
\begin{enumerate}
	\item $\sum_{a\in\mathcal A}\Fillvol(S_a)\leq C\cdot \area(S)^{3/2}$;
	\item $\sum_{b\in\mathcal B}\area(S_b)\leq q\cdot \area(S)$.
\end{enumerate}  
\eslem

\proof
The following three ingredients are needed.
\begin{itemize}
	\item By the coarea formula, if $x \in X$, then
\[\length(S\cap\d B(x,r))\leq\frac{d}{dr}\area(S\cap B(x,r))\]
for almost every $r > 0$.
\item For almost every point $p\in S^2$ we have 
\[\liminf\limits_{r\to 0}\frac{\area(S\cap B(f(p),r))}{r^2}\geq \pi.\]
\item For any Lipschitz $2$-sphere $\ph\colon S^2\to B(x,r)$, coning off at the point $x$ provides a Lipschitz $3$-ball $\bar\ph\colon B^3 \to B(x,r)$ with
\[\vol(\bar\ph)\leq \frac{r}{3}\cdot \area(\ph).\]
\end{itemize}
The proof of \cite[Proposition]{Wen-Gro} provides for every $\eps\in(0,1)$ finitely many disjoint balls $B_i:=B(x_i,r_i)$, $i=1,\ldots,k$, such that 
\begin{enumerate}
	\item $\area(S\cap B_i)\geq\frac{\pi}{4}\eps\cdot r_i^2$;
	\item $\length (S\cap\d B_i)\leq 2(\pi\eps)^{1/2}\cdot\area(S\cap B_i)^{1/2}$;
	\item $\sum_{i=1}^k\area(S\cap B_i)\geq\frac{1}{25}\cdot\area(S)$.
\end{enumerate} 

By Proposition~\ref{Pro:reg_fiber}, we may additionally assume that $f^{-1}(\d B_i)$ is a disjoint union of Jordan curves up to a $\Hm^1$-negligible set and that  
the gradient of the function $x\mapsto d(x_i,f(x))$ is non-zero for $\Hm^1$-almost every point in $f^{-1}(\d B_i)$.
We choose an auxiliary  $\rho\in(0,\frac12)$ small enough, such that the inflated balls $(1+\rho)B_i:=B(x_i,(1+\rho)r_i)$ are still disjoint.
Let $\hat F:=\{\hat\Ga_j\}$ be the family of Jordan curves contained in $\bigcup_{i=1}^k f^{-1}(\d B_i)$. Let $\hat\Ga_j$ be such a Jordan curve with $f(\hat\Ga_j)\subset \d B_i$. For one of the components of $S^2 \sm \hat\Ga_j$, denoted 
$\Om_j$, the gradient of $x\mapsto d(x_i,f(x))$ is an inner normal vector of 
$\Om_j$ at $\Hm^1$-almost all points of $\Ga_j$.
Since $f$ is Lipschitz continuous, only finitely many $\hat\Ga_j$ have such an outer
domain $\Om_j$ with $f(\Om_j)\not\subset (1+\rho)B_i$.
We remove all other Jordan curves from $\hat F$ and denote the remaining finite family by $F=\{\Ga_j\}$.  By Proposition~\ref{Pro:pre-homotopy}, we may assume that $F$
consists of bilipschitz Jordan curves.
Next, we fill every Lipschitz circle $f|_{\Ga_j}$ by a Lipschitz disc $D_j$ such that $\area(D_j)\leq\frac{1}{4\pi}\cdot\length(\d D_j)^2$. We denote by 
$\{D_j\}_{j\in J_i}$ the family of discs with boundary in $\d B_i$ and assume
without loss of generality that $D_j \sub B_i$.
By property (2), their total area is controlled by
\[
\sum_{j\in J_i}\area(D_j)\leq \frac{1}{4\pi}\cdot\length(S\cap\d B_i)^2\leq \eps\cdot\area(S\cap B_i).
\]
By inserting the discs $D_j$, we have decomposed $f$ into finitely many Lipschitz $2$-spheres 
\[S\cup\bigcup_j D_j=\bigcup_{a\in\mathcal A} S_a\cup\bigcup_{b\in\mathcal B} S_b\]
such that each $S_a$ lies in $(1+\rho)B_i$ for some $i=i(a)\in\{1,\ldots,k\}$ and each $S_b$ intersects the balls $B_i$ only in some of the inserted discs $D_j$.
Denote by $\mathcal A_i\subset\mathcal A$ the subset of spheres in $(1+\rho)B_i$. Using the coning inequality and the properties of our decomposition, we estimate
the filling volumes as follows:
\begin{align*}
&\sum_{a\in\mathcal A}\Fillvol(S_a)\leq\sum_i \sum_{a\in\mathcal A_i}\frac{(1+\rho)r_i}{3}\cdot\area(S_a)\\ 
&\quad\leq\sum_i \frac{(1+\rho)r_i}{3}\cdot\biggl(\area(S\cap (1+\rho)B_i)+\sum_{j\in J_i}\area(D_j)\biggr)\\
&\quad\leq \sum_i \frac{(1+\rho)r_i}{3}\cdot\bigl(\area(S\cap (1+\rho)B_i)+\eps\cdot\area(S\cap B_i)\bigr)\\
&\quad\leq \sum_i r_i\cdot\area\bigl(S\cap (1+\rho)B_i\bigr)\leq \sum_i \frac{2}{\sqrt{\pi\eps}}\cdot\area\bigl(S\cap (1+\rho)B_i\bigr)^{3/2}\\
&\quad\leq\frac{2}{\sqrt{\pi\eps}}\cdot\area(S)^{3/2}.
\end{align*}
Next, we estimate the total area of the remaining spheres:
\begin{align*}
&\sum_{b\in\mathcal B}\area(S_b)\leq \frac{24}{25}\cdot\area(S)+\sum_j \area(D_j)\\
&\quad\leq\frac{24}{25}\cdot\area(S)+\eps\cdot\sum_i\area(S\cap B_i)\leq \frac{49}{50}\cdot\area(S)%if \eps<1/13
\end{align*}
if $\eps\leq\frac{1}{50}$. This completes the proof of the Sublemma.
\qed
\medskip

To prove the theorem we will apply the Sublemma iteratively. Suppose that $S\subset B(o,R)$ for some $R>0$ and $o\in X$. 
Set $\la=q^{3/2}\in(0,1)$.
Choose $l\in\N$ such that $\frac{R}{3}\cdot\la^l\leq\area(S)^{1/2}$.
Applying the Sublemma $l$ times in $B(o,R)$, we can estimate the filling volume of $S$ as follows:
\begin{align*}
&\Fillvol(S)\leq\sum_{a\in\mathcal A_1}\Fillvol(S_a)+\ldots+\sum_{a\in\mathcal A_l}\Fillvol(S_a)+\sum_{b\in\mathcal B_l}\Fillvol(S_b)\\
& \leq C\cdot \biggl(\area(S)^{3/2}+\sum_{b\in\mathcal{B}_1}\area(S_b)^{3/2}+\ldots+\sum_{b\in\mathcal{B}_{l-1}}\area(S_b)^{3/2}
+\frac{R}{3}\cdot\sum_{b\in\mathcal B_l}\area(S_b)\biggr)\\
& \leq C\cdot \biggl(\area(S)^{3/2}+\la\cdot\area(S)^{3/2}+\ldots+\la^{l-1}\cdot\area(S)^{3/2}
+\frac{R}{3}\cdot\la^{l}\cdot\area(S)\biggr)\\
& \leq C\cdot \bigl(1+\la+\ldots+\la^{l-1}+1\bigr)\cdot\area(S)^{3/2}\\
&\leq C'\cdot\area(S)^{3/2}
\end{align*} 
with $C'=C\cdot\frac{2-\la}{1-\la}$, independent of $l$.
\qed

%%%%%%%%%%%%%%%%%%%%%%%%%%%%%%%%%%%%%%%%%%%%%%%%%%%%%%%%%%%%%%%%%%%%%%%%%%%%%%%

\section{Minimal triangles and tetrahedra} \label{sec:min_triangle}

We assume that $X$ is a proper $\CAT(0)$ space. Recall the notion of a minimal disc provided by Theorem~\ref{Thm:Plateau}.
For us, minimal discs filling geodesic triangles will play a special role, and we 
call such a disc a {\em minimal triangle}. 
The purpose of this section is to prove the linear isoperimetric inequality for 
minimal tetrahedra stated in Theorem~\ref{main_4}. A {\em minimal tetrahedron} is a Lipschitz $2$-sphere composed of four minimal triangles as faces, six geodesics as edges, and four points as vertices.

A key observation for the proof is that minimal triangles and 
tetrahedra in~$X$ have Nagata dimension at most $2$.
This is a consequence of the following a priori density bound for minimal triangles.

\begin{Lem} \label{Lem:triangle-density}
Let $u \colon B^2 \to X$ be a minimal triangle. For all $x \in X$ 
and $r > 0$,
\[
\area(u(B^2) \cap B(x,r)) \le \frac{3\pi}{2}\cdot r^2.
\]
\end{Lem}

\begin{proof}
By definition, $u$ is a solution to the Plateau problem for a geodesic triangle 
$\triangle\subset X$. If $\triangle$ is a Jordan curve, then the statement is a consequence of~\cite[Proposition~80]{Sta}
because geodesic triangles have total curvature at most $3\pi$.
However, the argument from \cite{Sta} also covers the case of a general triangle $\triangle$.
We recall the necessary steps, but refer the reader to \cite[Lemmata~78~and~79]{Sta} for more details. By Theorem~\ref{Thm:Plateau}, we can decompose $u$ as $u=u'\circ u''$  with $u''\colon B^2\to Z$ and $u'\colon Z\to X$. The space $Z$ is a $\CAT(0)$ disc retract and $u''$ is a surjective Lipschitz map with connected fibers. Furthermore, $u'$ is area minimizing among Lipschitz maps from $Z$, and preserves the length of every rectifiable curve. In particular, $u'$
is 1-Lipschitz and therefore $\d Z$ is a geodesic triangle
whose sides are mapped isometrically by $u'$.
We now extend $u'$ to a minimal plane in a larger $\CAT(0)$ space $\hat X$. We obtain $\hat X$ in two steps. We first glue a half-line to every vertex of $\triangle$ to obtain a $\CAT(0)$
space $X'$. Note that in $X'$ every side $s$ of $\triangle$
is contained in a complete geodesic $c_s$ that intersects $X$ precisely in $s$.
We then glue a flat half-plane along the boundary to every geodesic $c_s$ to obtain $\hat X$. By Reshetnyak's gluing theorem, $\hat X$
is $\CAT(0)$. In a similar way, we extend $Z$
to a $\CAT(0)$ plane $\hat Z$. Note that the map $u'$
extends to a map $f\colon \hat Z\to \hat X$ that is locally isometric on $\hat Z \setminus Z$.
The map $f$ is a {\em proper intrinsic minimal plane} in the sense of \cite[Definition~41]{Sta}. In particular, $f$
satisfies the monotonicity of area densities for all radii
$r>0$ \cite[Proposition~62]{Sta}. By construction, the area growth of $f$ is equal to the area growth of a flat cone over a circle of length $3\pi$:
\[\lim\limits_{r\to\infty}\frac{\area(f(\hat Z)\cap B(x,r))}{\pi r^2}=\frac{3}{2}.\]
Hence, we obtain the required inequality. 
\end{proof}

In what follows, we use the dimension theory definitions  as provided in Section~\ref{Sect:nag-dim}.

\begin{Lem} \label{Lem:triangle-dim}
Let $u \colon B^2 \to X$ be a minimal triangle. Then $u(B^2)$ has 
Assouad dimension at most $2$ and hence also Nagata dimension at most~$2$, both with absolute constants. 
\end{Lem}

\begin{proof}
Let $\eps \in (0,1)$, and let $Y \sub u(B^2)$ be a subset of diameter $s > 0$.
Put $r := \frac{\eps}{6}s$, and choose a maximal set $Z \sub Y$ of distinct points at mutual distance $> 3r$. 
The balls $B(z,3r)$ with $z \in Z$ cover $Y$, and their diameter is at most $\eps s$. 
Let $Z' \sub Z$ denote the set of points at distance at most $r$ from the geodesic triangle $u(S^1)$. 
Choosing a closest point projection $\pi \colon Z' \to u(S^1)$, we get a set $\pi(Z')$ of distinct points at mutual distance $> r$. 
Furthermore, the intersection of $\pi(Z')$ with each side of $u(S^1)$ has diameter at most $s$. 
It follows that
\[
\# Z' \le 3\Bigl( \frac{s}{r} + 1 \Bigr) \le \frac{21}{\eps} \le \frac{21}{\eps^2}.
\]
On the other hand, if $z \in Z \sm Z'$, then $\area(u(B^2) \cap B(z,r)) \ge \pi r^2$ by Theorem~\ref{Thm:Plateau}), 
and the $r$-neighborhood of $Z \sm Z'$ in $u(B^2)$ has diameter at most $s + 2r \le \frac43 s$ and area at most 
$\frac{3\pi}{2} \cdot (\frac43 s)^2 = \frac{8\pi}{3} s^2$ by Lemma~\ref{Lem:triangle-density}.
Hence,
\[
\#(Z \sm Z') \le \frac{8\pi s^2}{3\pi r^2} < \frac{96}{\eps^2}.
\]
This shows that $u(B^2)$ can be covered by $117 \eps^{-2}$ sets of diameter $\eps s$.
Thus, $u(B^2)$ has Assouad dimension at most~$2$, hence, by~\cite[Theorem~1.1]{LeDR}, also Nagata dimension at most~$2$, with an 
absolute implicit constant.
\end{proof}

We are now ready for the proof of Theorem~\ref{main_4}, which we restate for convenience.
The argument uses an idea of White~\cite{Whi-exist, Whi-hom}.

\begin{Thm} \label{Thm:tetra-filling} 
Let $X$ be a proper $\CAT(0)$ space  of asymptotic rank at most~$2$.
There exists a constant $\mu>0$ such that every minimal tetrahedron 
$f \colon S^2 \to X$ satisfies a linear isoperimetric inequality for fillings by balls,
\[
\Fillvol(f)\leq\mu\cdot\area(f).
\]
\end{Thm}

\begin{proof}
Let $f \colon S^2 \to X$ be a minimal tetrahedron with image $\tau$.
By definition, $S^2$ is parametrized by four positively oriented bilipschitz 
discs $\ph_1,\ldots,\ph_4 \colon B^2 \to S^2$ such that 
each $u_i := f \circ \ph_i \colon B^2\to X$ is a minimal triangle.
Accordingly, the integral cycle $S := f_\#(\d\bb{B^3}) \in \bI_{2,\cs}(X)$ induced by $f$ can be written as the sum $\sum_i u_{i\#}\bb{B^2}$.
By Lemma~\ref{Lem:triangle-density}, we have
\[
\bigl\| u_\#\bb{B^2} \bigr\|(B(x,r))\leq \frac{3\pi}{2}\cdot r^2
\]
for every minimal triangle $u$, every $x\in X$ and $r>0$. Thus, $S$ 
has uniformly controlled density. Hence, by Theorem~\ref{Thm:fillrad},
$S$ admits a filling by an integral $3$-current $V \in \bI_{3,\cs}(X)$ 
whose support lies within uniformly bounded distance $s > 0$ from $\tau$.
By Lemma~\ref{Lem:triangle-dim}, $\tau$ has Nagata dimension at most~$2$. Since $\tau$ is compact,
and by the remark preceding Lemma~\ref{Lem:complex}, there exists an absolute constant $c$ 
and a finite $cs$-bounded cover of the closed $s$-neighborhood $K$ of $\tau$ with $s$-multiplicity at most~$3$. 
Applying Lemma~\ref{Lem:complex} to $K$,
we find a finite simplicial complex $\Sig$ of dimension at most~$2$, metrized as a subcomplex of some simplex of edge length $s$ in a Euclidean space, and Lipschitz maps 
$\psi \colon K \to \Sigma$ and $\phi \colon \Sigma \to X$.
The push-forwards by $\psi$ of the filling $V$, and hence also of the cycle $\d V$, are zero
as currents. Since $\Sig$ is piecewise Euclidean, it is locally conical.
In particular, by \cite[Corollary~1.4]{RiSch_hom} or \cite[Theorem~1.3]{Mit_coin},  the homology theories induced by compactly supported integral currents and by Lipschitz chains, respectively, are isomorphic.
Thus, the Lipschitz cycle $\psi \circ f$ is homologically trivial.
Now we cone off the $1$-skeleton of $\Sigma$ to get a simply connected
complex $\hat\Sig$. We can arrange that $\hat\Sigma$ is again metrized as a subcomplex of some simplex of edge length $s$ in a Euclidean space. We extend $\phi$ to a Lipschitz map  
$\hat\phi\colon\hat\Sig\to X$, by first choosing an arbitrary image point for the new vertex 
of $\hat\Sigma$, and then extending to all of $\hat\Sigma$, via coning.
By Hurewicz, $\psi \circ f$ bounds a Lipschitz $3$-ball $\Psi$ in $\hat\Sig$,
which has volume zero. Then, $\hat\phi\circ\Psi$ is a Lipschitz $3$-ball of volume zero 
that fills $v \circ f$, where $v=\phi\circ\psi$. We combine $\hat\phi\circ\Psi$ with the geodesic homotopy $h$
from $f$ to $v \circ f$ to obtain a Lipschitz $3$-ball filling~$f$. 
By Lemma~\ref{Lem:complex}, the map  $v$ is $L$-Lipschitz where $L>0$
depends only on $c$,
and has controlled displacement, $d(x,v(x)) \le L\cdot s$ for all $x \in K$.
In particular,
$\vol(h)\leq L\cdot s \cdot \area(f)$.
Thus, the result holds for $\mu:=L\cdot s$.
\end{proof}

%%%%%%%%%%%%%%%%%%%%%%%%%%%%%%%%%%%%%%%%%%%%%%%%%%%%%%%%%%%%%%%%%%%%%%%%%%%%%%%

\section{Quantitative triangulation of discs}\label{Sec:triangulation}

In this section, the goal is to prove the necessary  quantitative
triangulation result for discs.
The arguments are inspired by \cite{Pap-AC, Pap-CC}.
Other quantitative triangulation results of surfaces have recently appeared in \cite{CrRo_trian, NtRo_approx}.
Our setting differs in that the underlying topology of the metric spaces in question may be complicated. To preserve the good topological properties of surfaces, we will instead work with pseudo-metrics.

Let $X$ be a complete metric space, $Z$ a complete length space homeomorphic to a closed disc and  $\varphi\colon Z\to X$ a
 Lipschitz map. We denote by $d_\ph$ the associated {\em intrinsic pseudo-metric} on $Z$,
\[d_\ph(x,y):=\inf_\ga\len(\ph\circ\ga)\]
where the infimum runs over all paths $\ga$ in $Z$ joining $x$ and $y$. 
We define the {\em intrinsic radius} of $Z$ by
\[
\rad_\ph(Z):=\max_{x\in Z}d_\ph(x,\d Z).
\]
Moreover, for a curve $\ga\subset Z$ we will use $\len_\ph(\ga):=\len(\ph\circ\ga)$
to denote its {\em $\ph$-length}.
 
By a {\em separating arc $\alpha$} of $Z$ we mean a
simple arc $\alpha \sub Z$ that intersects $\d Z$ precisely in the two
(distinct) endpoints.

The following proposition is analogous to the result for van Kampen
diagrams on p.~796 in~\cite{Pap-AC}.

\begin{Pro} \label{Pro:sausage}
Suppose that $\rad_\ph(Z) \le r \le \length_\ph(\d Z)/17$. Then there exists
a separating arc $\alpha$ of $Z$ such that the two discs $D_1,D_2$ with
$D_1 \cup D_2 = Z$ and $D_1 \cap D_2 = \alpha$ satisfy
\[
  \length_\ph(\d D_1) \le 17r, \quad
  \length_\ph(\d D_2) \le \length_\ph(\d Z) - r. 
\]
\end{Pro}

\begin{proof}
We denote by $d'_\ph$ the intrinsic pseudo-metric on $\d Z$ induced by $\ph|_{\d Z}$, thus
$d'_\ph(x,y) \ge d_\ph(x,y)$ is the shortest $\ph$-length of a subarc of $\d Z$ connecting 
$x,y\in\d Z$. Suppose first that there exists a point $p \in \d Z$ such that
\[
  q \in \d Z,\,d_\ph(p,q) \le 4r \Rightarrow d'_\ph(p,q) \le 6r.
\]
Let $\bar p \in \d Z$ be a point with
$d'_\ph(p,\bar p) = \frac12\length_\ph(\d Z)$, and note that $d_\ph(p,\bar p) > 4r$
because $d'_\ph(p,\bar p) > 6r$. Thus, the
set 
% $3r$-sphere 
$S^\ph_{3r}(p)=\{d_\ph(p,\cdot)= 3r\}$ separates $p$
from $\bar p$. Hence, there exists a simple arc $\beta$ in $S^\ph_{3r}(p)$ with endpoints
$p_1,p_2 \in \d Z$ that still separates $p$ from $\bar p$ in $Z$. 
Since $\rad_\ph(Z) \le r$, there exists a point $q \in \beta$ at $d_\ph$-distance 
at most $r$ from both connected components of $\d Z \sm \{p,\bar p\}$. 
For $i = 1,2$, choose a point $q_i$ closest to $q$ in the component containing $p_i$.
Since $d_\ph(p,q) = 3r$, we have $2r \le d_\ph(p,q_i) \le 4r$, hence
$2r \le d'_\ph(p,q_i) \le 6r$ and $4r \le d'_\ph(q_1,q_2) \le 12r$.
There exists a separating arc from $q_1$ to $q_2$ of $\ph$-length less than or
equal to $2r$, and the resulting discs $D_1$ and $D_2$ containing
$p$ and $\bar p$, respectively, satisfy
\[
  \length_\ph(\d D_1) \le 14r, \quad
  \length_\ph(\d D_2) \le \length_\ph(\d Z) - 4r + 2r.
\]

In the remaining case, we have that for every $p \in \d Z$ there exists
a $q \in \d D$ with $d_\ph(p,q) \le 4r$ and $d'_\ph(p,q) \ge 6r$. We fix 
$p,q$ such that $d'_\ph(p,q)$ is minimal among all such pairs.
If $d'_\ph(p,q) \le 9r$, we choose a separating arc from $p$ to $q$ of $\ph$-length
less than $5r$ to get a decomposition with
\[
  \length_\ph(\d D_1) \le 14r, \quad
  \length_\ph(\d D_2) \le \length_\ph(\d Z) - 6r + 5r.
\]
If $d'_\ph(p,q) > 9r$, let $\beta \sub \d Z$ be a subarc from $p$ of $q$
of $\ph$-length $d'_\ph(p,q)$, and let $p' \in \beta$ be the point with
$d'_\ph(p,p') = 9r$. There exists a $q' \in \d Z$ with $d_\ph(p',q') \le 4r$ and
$d'_\ph(p',q') \ge d'_\ph(p,q)$ by the choice of $p$ and $q$. In particular,
$q' \not\in \beta$, thus any path in $Z$ from $p$ to $q$ separates
 $p'$ and $q'$. It follows that 
\[
d_\ph(p,p') < d_\ph(p,q)+d_\ph(p',q')\leq 8r,
\]
thus there exists a separating arc from $p$ to $p'$ of $\ph$-length
less than $8r$. This gives a decomposition satisfying the assertion of the lemma.
\end{proof}  

Recall the notion of a  disc decomposition from Section~\ref{Sec:gen_discs}.

\begin{figure}[ht]
    \centering
    \includegraphics[scale=0.35,trim={0cm 0cm 7cm 0cm},clip]{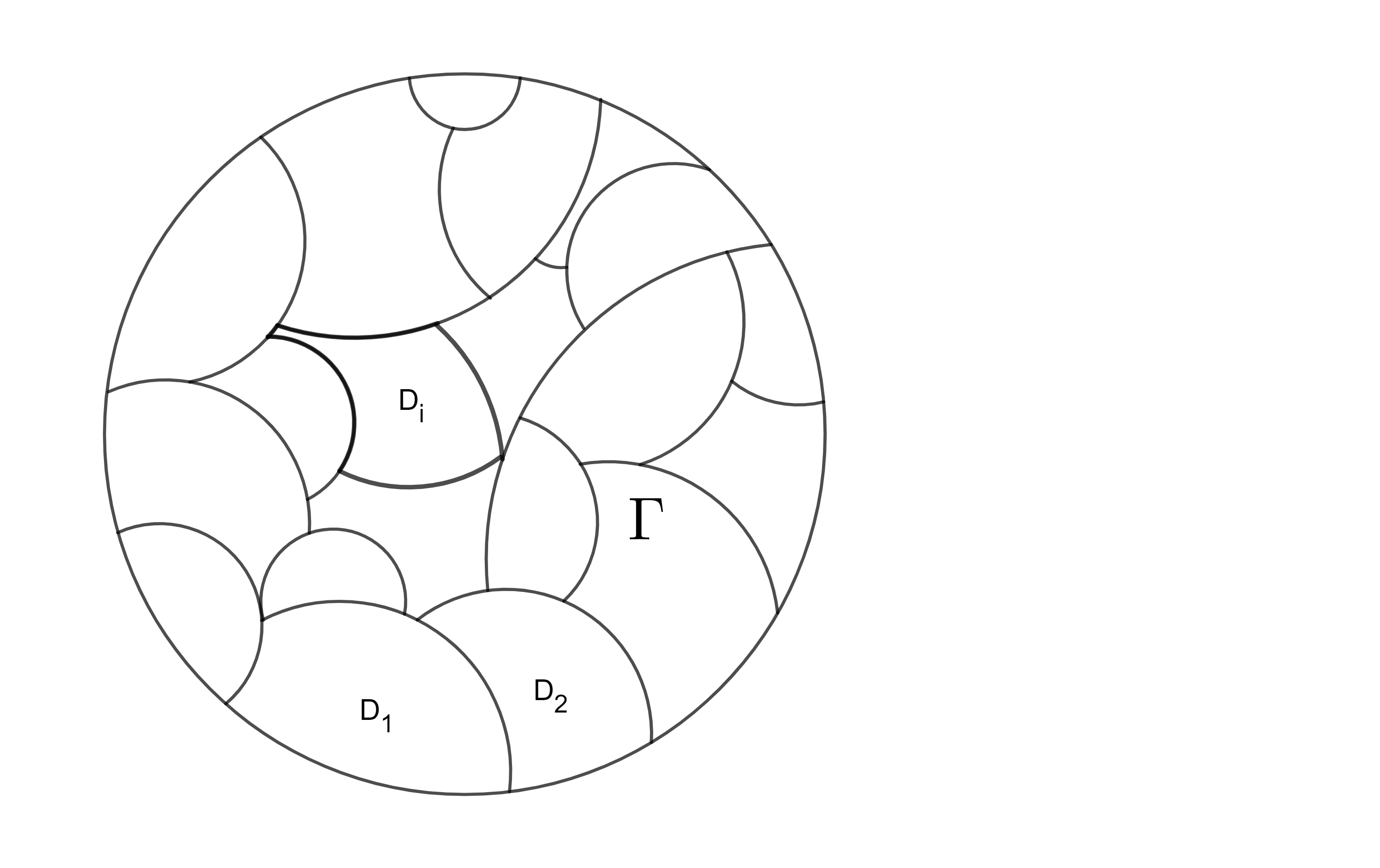}
    \caption{Example of a disc decomposition.}
    \label{fig:disc_decom}
\end{figure}

\begin{Cor}\label{Cor:sausage}
Let $r \ge \rad_\ph(Z)$. Then there exists a disc decomposition
$Z = D_1 \cup \ldots \cup D_n$ induced by an embedded graph $\Ga\subset Z$ such that
\begin{enumerate}
\item[\rm (1)] $n \le \length_\ph(\d Z)/r$, provided that $\length_\ph(\d Z) \ge r$;
\item[\rm (2)] $\length_\ph(\d D_i) \le 17r$ for $i = 1,\ldots,n$.  
\end{enumerate}
Furthermore, $D_i \cap D_j$ is either empty, a point, or a simple arc
for $i \ne j$.
\end{Cor}

\begin{proof}
If\/ $\length_\ph(\d Z) \le 17r$, then the result holds for $n = 1$.  
If\/ $\length_\ph(\d Z) > 17r$, we use Proposition~\ref{Pro:sausage} and induction.
\end{proof}

%%%%%%%%%%%%%%%%%%%%%%%%%%%%%%%%%%%%%%%%%%%%%%%%%%%%%%%%%%%%%%%%%%%%%%%%%%%%%%%

\section{Improving 2-spheres}\label{Sec:imp}

This is a technical section that provides a tool needed in the  proof of our main result.
The setting we have in mind here is that we already know how to find fillings of small volume for all Lipschitz $2$-spheres
of area at most $A>0$.
Our aim is then to understand how to improve the intrinsic structure of a Lipschitz $2$-sphere $S$ 
of area only a little bit larger than $A$.

We begin with a useful calculus fact.

\begin{Lem} \label{Lem:convex}
Given $\del > 0$, there exists $\lam = \lam(\del) \in (0,1)$
such that
\[
 (x+y)^{1+\del} > (x+a)^{1+\del} + (y+a)^{1+\del}    
\]
whenever $0 < x \le y$ and $a \le \lam x$.
\end{Lem}

\begin{proof}
Dividing the desired inequality by $(y+a)^{1+\del}$, and setting
$a := \lam x$ and $q := x/(y + a)$, we see that it suffices to 
find $\lam$ so that
\[
\bigl( (1 - \lam)q + 1 \bigr)^{1+\del}
> \bigl( (1 + \lam)q \bigr)^{1+\del} + 1.
\]
Since $(1+a)^{1+\del} \ge 1 + (1+\del)a$ for $a \ge 0$, 
and since $(1 + \lam)q \le 1$, this can be further reduced to
the inequality
\[
(1 + \del)(1 - \lam) > 1 + \lam, 
\]
which holds for $\lam < \del/(2+\del)$.
\end{proof}

Let again $X$ be a $\CAT(0)$ space.
For $M, \del >0$ we consider all Lipschitz spheres $f\colon S^2 \to X$ such that
\begin{equation*}
\Fillvol(f)\geq M\cdot \area(f)^{1+\del}.
\label{eq:FVgeqdel}
\end{equation*}
For $\eps\geq 0$ we call such a sphere {\em $(M,\delta,\eps)$-minimal}\/ if
every Lipschitz sphere $f' \colon S^2 \to X$ with 
$\area(f')\leq\area (f)-\eps$ satisfies the reverse inequality
$\Fillvol(f')\leq M\cdot \area(f')^{1+\del}$.

This notion is related to the intrinsic isoperimetric profile of a sphere. We say that a Lipschitz sphere $f\colon S^2\to X$ is {\em intrinsially $(c,\theta)$-isoperimetric}, for some $c>0$ and $\theta \ge 0$, if every disc $D \sub S^2$ with $\area_f(D) \le \frac12\area(f)$ and $\area_f(D) \ge \theta$ satisfies 
$\area_f(D) \leq c\cdot\length_f(\d D)^2$.

\begin{Lem} \label{Lem:round-minimal} 
For every $\delta > 0$ and $\theta \ge 0$ there exist $c=c(\del)>0$ and $\eps(\del,\theta)\in [0,\theta]$ such that the following holds for every $M > 0$.
If $f\colon S^2\to X$ is an $(M,\delta,\eps)$-minimal
Lipschitz sphere, then $f$ is intrinsically $(c,\theta)$-isoperimetric.
\end{Lem}

\begin{proof}
Given $\del > 0$ and $\theta \ge 0$,
let $\lam = \lam(\del) \in (0,1)$ be the constant from Lemma~\ref{Lem:convex} and put 
\[
c := \frac{1}{4\pi\lam}, \quad \eps: = (1-\lam)\,\theta.
\]
Let $f \colon S^2 \to X$ be an $(M,\del,\eps)$-minimal Lipschitz sphere and suppose, contrary to the assertion, 
that there exists a decomposition $S^2 = D_1 \cup D_2$ into two discs along a Jordan curve $\alpha \sub S^2$ such that 
\[
\theta \leq x := \area_f(D_1) \leq y := \area_f(D_2)
\]
and $\area_f(D_1) > c\cdot \length_f(\alpha)^2$. After reparametrizing $f$ using 
Proposition~\ref{Pro:pre-homotopy}, we can assume that $\alpha$ is piecewise geodesic. 
Then, by Proposition~\ref{Pro:reshetnyak} and the choice of $c$, we can fill the closed 
curve $f|_\alpha$ by a Lipschitz disc of area
\[
a \leq \frac{1}{4\pi}\cdot\length_f(\alpha)^2 \leq \lam\cdot\area_f(D_1) = \lam x
\]
in $X$. Since $\alpha$ is piecewise geodesic, this creates two Lipschitz $2$-spheres $f_1,f_2$ with
\[
\area(f) = x+y, \quad \area(f_1) = x+a, \quad \area(f_2) = y+a.
\]
Note that $x - a \ge (1 - \lam)\,x \ge (1 - \lam)\,\theta = \eps$ and hence
\[
\area(f_i)\leq \area (f)-x+a\leq\area(f)-\eps.
\]
Thus, by $(M,\delta,\eps)$-minimality of $f$ and
Lemma~\ref{Lem:convex}, we have  
\begin{align*}
\Fillvol(f) &\leq \Fillvol(f_1) + \Fillvol(f_2) \\
&\leq M\cdot (x+a)^{1+\delta} + M\cdot(y+a)^{1+\delta} \\ 
&< M\cdot (x+y)^{1+\delta}=M\cdot \area(f)^{1+\delta},  
\end{align*}
contradicting the assumption on $f$.
\end{proof}

We will now work towards the main result of this section, Proposition~\ref{Pro:triangulation}. We begin with a modified version of the classical Besicovitch inequality \cite{Gro_FRR}.
In the following we denote by $Q=[0,1]^2$ the unit square.

\begin{Lem}[Besicovitch] \label{Lem:besi} Let $X$ be a complete metric space and $\ph\colon Q\to X$ a Lipschitz disc. Suppose that the opposite sides of the square 
have intrinsic distance at least  $a$ and $b$, respectively.
Then $\area(\ph)\geq a\cdot b$.
\end{Lem}

\proof
Let us first assume that $\ph$ is bilipschitz. Denote by $Z=(Q,d_\ph)$ the
associated intrinsic disc and by $\pi_\ph\colon Q\to Z$ the induced canonical 
projection which is also bilipschitz. By \cite[Corollary~3.2]{LWint}, the metric derivatives of $\ph$
and $\pi_\ph$ coincide almost everywhere. Thus, $\area(\ph)=\area(\pi_\ph)$.
Since $\ph$ is bilipschitz, the area formula implies $\area(\pi_\ph)=\Hm^2(Z)$.
Hence, we obtain the claim from  the classical Besicovitch inequality which shows 
$\Hm^2(Z)\geq a\cdot b$.
If $\ph$ is only Lipschitz, we consider for $\eps>0$ the auxiliary bilipschitz map
$\ph_\eps\colon Q\to X\times\R^3$ defined by $\ph_\eps(x)=(\ph(x),\eps x)$.
Note that $d_{\ph_\eps}\geq d_\ph$ and $\area(\ph_\eps)\to\area(\ph)$ for $\eps\to 0$.
This concludes the proof.
\qed

\medskip

The proof of Proposition~\ref{Pro:triangulation} also employs the following decomposition lemma for spheres analogous to~\cite[Proposition~2.3]{Pap-CC}. 
Since our setting is slightly different, we provide a complete argument.

Similarly as for spheres, we say that a Lipschitz disc $f \colon D \to X$ is {\em intrinsically $(c,\theta)$-isoperimetric} if every subdisc
$D' \sub D$ with $\area_f(D') \ge \theta$ 
satisfies $\area_f(D') \leq c\cdot\length_f(\d D')^2$.

\begin{Lem} \label{Lem:cut} Let $f\colon S^2\to X$ be a Lipschitz $2$-sphere. 
Then there exists a simple closed curve $\alpha \sub S^2$ such that the
two discs $H_1,H_2$ with $H_1 \cup H_2 = S^2$ and $H_1 \cap H_2 = \alpha$
satisfy
\[
\area_f(H_i) \ge \frac14 \area(f) \;\;(i = 1,2), \quad
\length_f(\alpha) \le 4\sqrt{\area(f)}.
\]
Moreover, if $f$ is intrinsically $(c,\theta)$-isoperimetric for some $c > 0$ and $\theta \in [0,\frac14\area(f)]$,
then both $f|_{H_1}$ and $f|_{H_2}$ are intrinsically $(3c,\theta)$-isoperimetric.
\end{Lem}

\proof
Since $f$ is Lipschitz, there exists a Jordan curve $\al$ in $S^2$ such that the resulting discs $H_1$ and $H_2$ both have $f$-area at least $\frac14 \area(f)$. Denote by $\mathcal{F}$
the family of all such Jordan curves in $S^2$ and let $\bar L$
denote the infimal $f$-length of curves in $\mathcal{F}$. Let $\al\in\mathcal{F}$
be $\eps$-minimal, i.e.~$\length_f(\al)<\bar L+\eps$. Denote by $H_1$ and $H_2$
the discs associated to $\al$ and assume that $\area_f(H_1)\geq\area_f(H_2)$.
Note that if $c$ is an arc in $H_1$ joining two points on $\al$, and $\bar c$
is the shortest subarc of $\al$ joining the same points, then $\length_f(c)\geq\length_f(\bar c)-\eps$.
This shows that if we decompose $\al$ into four arcs of equal length, then opposite arcs 
have distance in $H_1$ at least $\frac14\length_f(\al)-\eps$. By Lemma~\ref{Lem:besi}, we conclude that
\[
\area_f(H_1)\geq\Bigl(\frac14\length_f(\al)-\eps\Bigr)^2.
\]
Choosing $\eps<\bigl(1-\frac{\sqrt{3}}{2}\bigr)\sqrt{\area(f)}$, we obtain the desired inequalities.

Suppose now that $f$ is intrinsically $(c,\theta)$-isoperimetric with $c > 0$ and $\theta \in [0,\frac14 \area(f)]$, and $D \sub H_i$ is a subdisc. We only need to consider the case 
that $\area_f(D) > \frac12\area(f)$. Then
\[
\area_f(D) \le \frac{3}{4} \area(f) \le 3 \area_f(S^2 \sm D) \le 3c \cdot \length_f(\d D)^2
\]
as claimed.
\qed

\medskip

Now we come to the main goal of this section. The result is inspired by~\cite{Pap-AC} (see p.~803 therein).

\begin{Pro} \label{Pro:triangulation}
Let $X$ be a proper $\CAT(0)$ space. 
For $\delta,\eps \in (0,1]$ and $M>0$, let $f\colon S^2\to X$ be an $(M,\delta,\eps)$-minimal Lipschitz $2$-sphere with $\area(f)\geq 4$. 
Suppose further that $f$ is intrinsically $(c,\theta)$-isoperimetric for some 
$\theta\in[0,1]$ and $c \ge 1$. 
Then we find for every $s>0$ with $\frac{\theta}{3c} \le s^2 \le \frac{1}{8c}\area(f)$ a Lipschitz $2$-sphere $f'\colon S^2\to X$ with $\area(f')\leq\area(f)$ and a Lipschitz homotopy $h$ from $f$ to $f'$ with
\[
\vol(h)\leq 12 M\cdot c\cdot s^{2\delta}\cdot\area(f);
\]
moreover, $f'$ admits a  disc decomposition $S^2 = D_1 \cup \ldots \cup D_n$, induced by a finite Lipschitz graph $\Ga\subset S^2$, such that
\begin{enumerate}
\item[\rm (1)] $n \le 14\cdot\area(f)/s^2$,
\item[\rm (2)] $\length_{f'}(\d D_i) \le 34 s$,  
\item[\rm (3)] $\area_{f'}(D_i) \le 3c \cdot (34 s)^2$ 
\end{enumerate}  
for $i = 1,\ldots,n$.
In particular, $\length_{f'}(\Ga)\leq 238 \cdot\area(f)/s$.
\end{Pro}

\begin{proof}
Let $\alpha$ and $H_1, H_2$ be given as in Lemma~\ref{Lem:cut}.
Then
\[
\length_f(\d H_i) \leq 4 \sqrt{\area(f)}\le 8 \sqrt{\area_f(H_i)}, \quad i=1,2.
\]
We will provide the necessary structure by treating the discs  $f|_{H_i}$, $i=1,2$, individually.
Let $\ph \colon D \to X$ be either one of these discs.

Consider the Lipschitz functions $g\colon D\to\R$ and
$g_s \colon D \to \R/s\Z$ defined by $g(x) := d_\ph(x,\d D)$ 
and $g_s(x) := g(x) \bmod s$.
By Proposition~\ref{Pro:reg_fiber}, we find $s_1\in(\frac{1}{2}s,s)$ such that the $s_1$-fiber satisfies $g_s^{-1}(s_1)=\hat\Pi\cup N$
where $\hat\Pi$ is a disjoint union of Jordan curves 
$\hat\Pi = \bigcup_{j \in J} \alpha_j$, $J \sub \N$, 
and $N$ is $\mathcal H^1$-negligible. Moreover,
by the coarea inequality, we can choose $s$ such that 
\[
\length_\ph(\hat\Pi)\leq\frac{2\area(\ph)}{s}.
\] 
We let $\Om_j\subset D$ denote the Jordan domain 
of $\alpha_j$, and we define $J_s$ as the set of all $j\in J$
such that $\alpha_j$ does not lie in the closure of a domain $\Om_k$ with $\length_\ph(\alpha_k)<s$.
We assume $J_s=\{1,\ldots,m\}$ and set
\[
\Pi=\bigcup_{j \in J_s} \alpha_j.
\]

In particular, for each $j\in J_s$ the curve $\alpha_j$ itself has 
$\length_\ph(\alpha_j)\geq s$.
Hence, 
\[
m\leq \frac{2\area(\ph)}{s^2}.
\]
Define $s_0=0$ and $s_i=s_1+(i-1)s$ for $i\geq 2$. 
We have $g_s^{-1}(s_1)=\bigcup_{i=1}^\infty g^{-1}(s_i)$. 
The regions $g^{-1}[s_{i-1},s_i]$ have $\ph$-radius at most $s$.

By Proposition~\ref{Pro:pre-homotopy},
we can deform  $\Pi$ to be piecewise geodesic without changing its $\ph$-length, at the cost of precomposing $f$ with a monotone Lipschitz map that is uniformly close to the identity map on $S^2$. 
This deformation  requires a Lipschitz homotopy of zero volume.
In particular, the resulting Lipschitz sphere is still intrinsically $(c,\theta)$-isoperimetric and $(M,\delta,\eps)$-minimal.
Hence, we can assume that $\Pi$ itself consists of piecewise geodesic Jordan curves.

We now define $\ph'\colon D\to X$ such that $\ph'=\ph$
on $D\setminus\bigcup_{j \in J \sm J_s} \Om_j$ and $\ph'|_{\Om_j}$ is a minimal disc if $j \not\in J_s$ and $\Om_j$ is not contained in a larger Jordan domain 
$\Om_k$ with $k \not\in J_s$. In particular, for such a minimal disc $\ph'|_{\Om_j}$ we have  $\area_{\ph'}(\Om_j)\leq\frac{1}{4\pi}s^2$ and 
hence $\rad_{\ph'}(\Om_j)\leq\frac{1}{2\pi}s$ 
by Proposition~\ref{Pro:reshetnyak} and~\cite[Lemma~51]{Sta}. 
We conclude that the $\ph'$-radii of the components of  $D\setminus \Pi$ are bounded by $\frac{2\pi+1}{2\pi}s \le 2s$.

The closure of each component of $D\setminus \Pi$
is a disc with holes. By Lemma~\ref{Lem:2nd_path},
after a small deformation, we may assume that  there are $m$ pairs of 
disjoint piecewise geodesic arcs $\beta_j^\pm$ in $D$ of $\ph'$-length at most $2s$ such that the graph $\tilde\Ga=\Pi\cup\bigcup_{j=1}^m(\beta_j^-\cup\beta_j^+)$
induces a disc decomposition of $D$ into discs $D_l$
of $\ph'$-radius at most $2s$.

\begin{figure}[ht]
    \centering
    \includegraphics[scale=0.35,trim={0cm 0cm 3cm 2.5cm},clip]{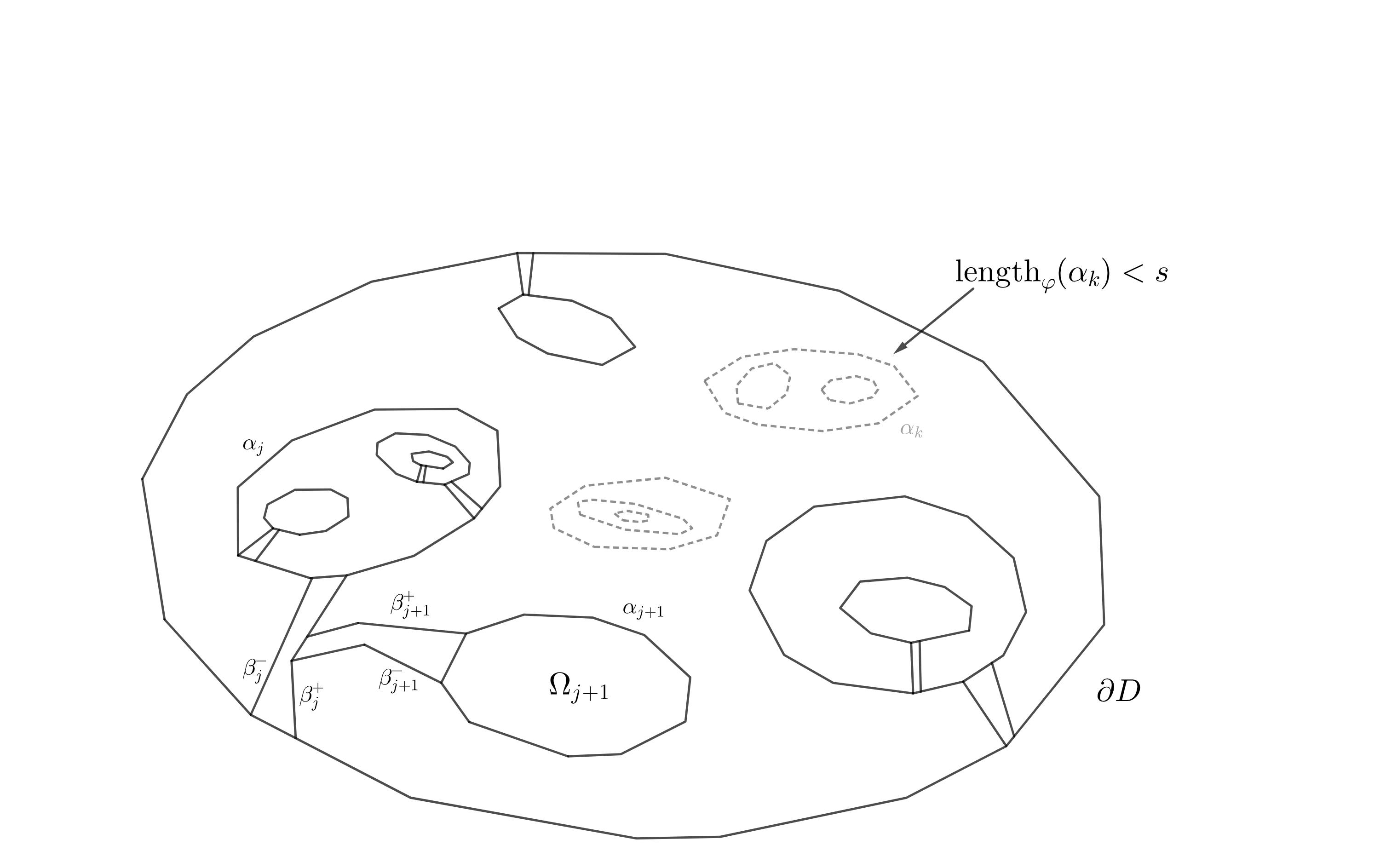}
    \caption{Illustration of the decomposition process.}
    \label{fig:decom_pro}
\end{figure}

Now we apply Corollary~\ref{Cor:sausage} with $r=2s$ to each of those discs.
Thus, either $\length_\ph(\d D_l)< 2s$ and we set $O_l:=D_l$, or $\length_\ph(\d D_l)\geq 2s$ and we obtain a  decomposition into  discs 
\[D_l=O_{1l}\cup\ldots\cup O_{pl}\]
induced by a graph $\Ga_l\subset D_l$ with $p\leq \frac{1}{2s}\length_{\ph'}(\d D_l)$ and $\length_{\ph'}(\d O_{il})\leq 34s$.
Renumbering, we obtain a disc decomposition
\[D=O_1\cup\ldots\cup O_{\tilde n}\]
induced by a graph $\hat\Ga=\tilde\Ga\cup\bigcup_l \Ga_l$.
Since there are $2m+1$ discs $D_l$, and since $\area(\ph) \ge \area(f)/4 \ge 2s^2$, we have
\begin{align*}
\tilde n &\leq\frac{1}{2s}\cdot\sum_l \length_\ph(\hat\d D_l)+(2m+1)\\
&\leq \frac{1}{s}\Bigl(\length_{\ph}(\Pi)+m\cdot 2s+\frac{1}{2}\cdot\length_\ph(\d D)\Bigr)+(2m+1)\\
&\leq \frac{1}{s}\left(\frac{2\area(\ph)}{s}+\frac{4\area(\ph)}{s}+4\sqrt{\area(\ph)}\right)+\frac{5\area(\ph)}{s^2}\\
&\leq \frac{14 \area(\ph)}{s^2}.
\end{align*}
Hence, by applying the above construction to both discs $f|_{H_i}$, $i=1,2$, we obtain a disc decomposition 
\[ 
S^2=D_1\cup\ldots D_n
\]
induced by $\Ga=\hat\Ga_1\cup\al\cup\hat\Ga_2$, where
$\hat\Ga_i$ denotes the graph constructed for $f|_{H_i}$.
(If it happens that $\Ga = \al$, then we introduce an arbitrary vertex 
on $\alpha$, so that $\Ga$ is a graph.)
We have already established properties~(1) and~(2). The remaining property~(3)
follows because $\theta \le \area(f)/4$, and therefore 
both $f|_{H_1}$ and $f|_{H_2}$ are intrinsically $(3c,\theta)$-isoperimetric 
by Lemma~\ref{Lem:cut}.

It remains to construct the homotopy $h$. Again, we consider $f|_{H_1}$ and $f|_{H_2}$
separately and continue with the above setup. 
For each of the minimal discs $\ph'|_{\Om_j}$ in the definition of $\ph'$, 
the union with $\ph|_{\Om_j}$
defines a Lipschitz $2$-sphere $\sig_j$ with 
\[
\area(\sig_j)\leq 2\area_\ph(\Om_j)\leq 6c s^2\le\frac{3}{4}\cdot\area(f)\leq\area(f)-\eps\]
by the properties of $f$ and the choice of $s$. Thus, by assumption, we have
\[
\Fillvol(\sig_j)\leq M\cdot\area(\sig_j)^{1+\delta}.
\]
Filling all spheres $\sig_j$, we obtain $h$ and the bound
\begin{align*}
\vol(h)&\leq M\cdot \sum_j \area(\sig_j)^{1+\delta}
\leq M\cdot \sum_j (2\cdot\area_\ph(\Om_j))^{1+\delta} \\
&\leq 4M\cdot \sum_j \area_\ph(\Om_j)\cdot(3c s^2)^{\delta}
\leq 12M\cdot c\cdot s^{2\delta}\cdot \area(\ph).
\end{align*}
Note that $h$ keeps the boundary of $\ph$ fixed. Applying the construction to both 
$f|_{H_1}$ and $f|_{H_2}$, we get the desired homotopy of $f$.
This completes the proof.
\end{proof}

%%%%%%%%%%%%%%%%%%%%%%%%%%%%%%%%%%%%%%%%%%%%%%%%%%%%%%%%%%%%%%%%%%%%%%%%%%%%%%%%%%%%%%%%%%%%%%%%%%%%%%%%%%%%%%%%%%%%%%%%%%%%%

\section{Proof of main theorem}\label{Sec:proof}

In this section we will prove Theorem~\ref{main_1} and Corollary~\ref{Cor:filling_surfaces}.

By definition, the asymptotic rank of a $\CAT(0)$ space $X$ equals the supremum of all $n$ for which there exist sequences of positive numbers $\lam_k \to \infty$, $\eps_k \to 0$, and maps $\phi_k$ from the Euclidean unit ball $B^{n} \sub \R^n$ into $X$ with 
\[
\bigl| \lam_k^{-1} d(\phi_k(y),\phi_k(y')) - \|y-y'\| \bigr| \le \eps_k
\]
for all $y,y' \in B^n$. This means that the sets $\phi_k(B^n)$, equipped with (the restrictions of) the rescaled metrics $\lam_k^{-1}d$, converge in the Gromov--Hausdorff topology to $B^n$.

Recall further that $\om_n$ denotes the Lebesgue measure of $B^n$. We let 
\[
  \ic_n := \frac{\vol_n(B^n)}{\vol_{n-1}(S^{n-1})^{n/(n-1)}}
  = \frac{\om_n}{(n\,\om_{n})^{n/(n-1)}} = \frac{1}{(n^n \om_n)^{1/(n-1)}}
\]
denote the constant in the optimal isoperimetric inequality in $\R^n$. In particular $\ic_2 = 1/(4\pi)$ and $\ic_3 = 1/(6\sqrt{\pi})$.

We start with the proof of the implication $(1) \Rightarrow (3)$ in Theorem~\ref{main_1},
which holds more generally in all dimensions. This uses similar arguments 
as in~\cite[Sect.~4]{Wen-EII}.

\begin{Pro} \label{Pro:AR}
Let $X$ be a proper $\CAT(0)$ space.
Suppose that there exist constants $c < \ic_{n+1}$ and $v>0$ such that every Lipschitz $n$-sphere $S \sub X$
with $\vol_n(S) > v$ bounds a Lipschitz $(n+1)$-ball $B \sub X$ with $\vol_{n+1}(B)\leq c\cdot \vol_n(S)^{(n+1)/n}$.
Then the asymptotic rank of\/ $X$ is at most~$n$.
\end{Pro}

\begin{proof}
Suppose for contradiction that the asymptotic rank of $X$ is at least $n+1$.
Then there exists a sequence of $(1 + \eps_k)$-bi-Lipschitz embeddings
\[
\phi_k \colon B^{n+1} \cap (\eps_k\Z)^{n+1} \to X_k := (X,\lam_k^{-1} d),
\]
where $\lam_k \to \infty$ and $\eps_k \to 0$. By the generalized Kirszbraun theorem~\cite{LS_kirszbraun},
every $\phi_k$ can be extended to a $(1 + \eps_k)$-Lipschitz map
$\ph_k \colon B^{n+1} \to X_k$. Define $S_k := (\ph_k)_\#\,\d \bb{B^{n+1}} \in \bI_{n,\cs}(X_k)$,
and note that $\spt(S_k) \sub \ph_k(B^{n+1}) \sub B(x_k,2)$ for $x_k := \ph_k(0)$.
By Theorem~\ref{Thm:plateau-n} there exists a minimizing filling $V_k \in \bI_{n+1,\cs}(X_k)$ 
of $S_k$ with support in $B(x_k,2)$.
It follows from the assumption that for sufficiently large $k$, 
\[
  \M(V_k) \le c \cdot \M(S_k)^{(n+1)/n}
  \le c (1 + \eps_k)^{n+1}\cdot \vol_n(S^n)^{(n+1)/n}.
\]
Since the $\ph_k$ are uniformly Lipschitz, and by the lower density bound for $V_k$,
the sets $K_k := \ph_k(B^{n+1}) \cup \spt(V_k)$ are equi-compact. Hence, by Gromov's
theorem~\cite{Gro_polygrowth}, they admit isometric embeddings $\psi_k \colon K_k \to Z$
into a single compact metric space $Z$. After passing to subsequences, we may
assume that the $\psi_k(K_k)$ converge to a compact set $K \sub Z$ with respect
to the Hausdorff distance in $Z$, the $(\psi_k)_\#V_k$ converge weakly to a current
$V \in \bI_{n+1,\cs}(Z)$ with support in $K$ (recall Theorem~\ref{Thm:cptness}), 
and the $\psi_k \circ \ph_k$
converge uniformly to an isometric embedding $\iota \colon B^{n+1} \to Z$
with image in $K$. This implies that
$(\psi_k \circ \ph_k)_\#\bb{B^{n+1}} \to T := \iota_\#\bb{B^{n+1}} \in \bI_{n+1,\cs}(Z)$
weakly. Note that $\d T = \d V$, as $\d(\psi_k \circ \ph_k)_\#\bb{B^{n+1}} =
(\psi_k)_\# S_k = \d(\psi_k)_\#V_k$. By the lower semi-continuity of mass
under weak convergence, and since $c < \ic_{n+1}$, we have
\[
  \M(V) \le \liminf_{k \to \infty} \M(V_k) \le c \cdot \vol_n(S^n)^{(n+1)/n} < \om_{n+1}.
\]
Note that $K$ is the Gromov--Hausdorff limit of the sets $K_k \sub X_k$ and thus embeds isometrically
into any ultralimit $(X_\om,x_\om)$ of the $(X_k,x_k)$. Since $(X_\om,x_\om)$ is a $\CAT(0)$
space, and the image of $\iota(B^{n+1})$ is a closed convex subset, 
we infer that there exists a $1$-Lipschitz retraction $\pi \colon K \to \iota(B^{n+1})$.
The cycle $\pi_\#V - T$ vanishes, because any $(n+2)$-dimensional filling
in $\iota(B^{n+1})$ is zero. Hence, $\M(T) = \M(\pi_\#V) \le \M(V) < \om_{n+1}$. This contradicts
the fact that $\M(T) = \M(\iota_\#\bb{B^{n+1}}) = \om_{n+1}$.
\end{proof}

It is clear that  $(2)$ implies $(1)$ in Theorem~\ref{main_1}.
The implication $(3)\Rightarrow (2)$ is more elaborate and we need
some preparation. The following lemma  allows us to
deform a general Lipschitz disc to a {\em piecewise minimal\/} Lipschitz disc, while keeping control on boundary length 
and area.

\blem\label{Lem:triangular_disc}
Let $X$ be a $\CAT(0)$ space and $\ph\colon D\to X$ a Lipschitz disc with $\area(\ph)\leq a\cdot t^2$
and $\length_\ph(\d D)\leq t$ for some constants $a, t>0$.
Let $V=\{v_1,\ldots,v_n\}\subset\d D$ for $n \ge 1$ be a cyclically ordered set and denote by $e_i\subset\d D$ the arc between $v_i$ and $v_{i+1}$, where $v_{n+1} := v_1$.
Moreover, let $D_i\subset D$ denote the closed disc bounded by $e_i$ and the geodesic $v_iv_{i+1}$.
Let $\psi\colon D\to X$ be a Lipschitz disc such that
\begin{itemize}
\item $\psi$ agrees with $\ph$ on $V$;
    \item $\psi$ maps the whole disc $D_i$ to the geodesic in $X$ from $\ph(v_i)$ to $\ph(v_{i+1})$;
    \item $\psi$ maps every geodesic $v_1 v_i$ to the geodesic in $X$ from $\ph(v_1)$ to $\ph(v_{i})$;
    \item $\psi$ restricts to a minimal triangle $T_i$ on every triangle
    $v_1v_iv_{i+1} \subset D$.
\end{itemize}
Then there exists a Lipschitz homotopy $h:S^1\times[0,1]\to X$ from $\ph|_{S^1}$ to $\psi|_{S^1}$ with 
$\area(h)\leq\frac{1}{\pi}\cdot t^2$, and the discs $\ph$ and $\psi$ together with the homotopy $h$
form a Lipschitz $2$-sphere $\sigma \colon S^2\to X$ with $\area(\sigma)\leq (a+1)\cdot t^2$. 
\elem

\begin{figure}[ht]
    \centering
    \includegraphics[scale=0.3,trim={0cm 0cm 7cm 0cm},clip]{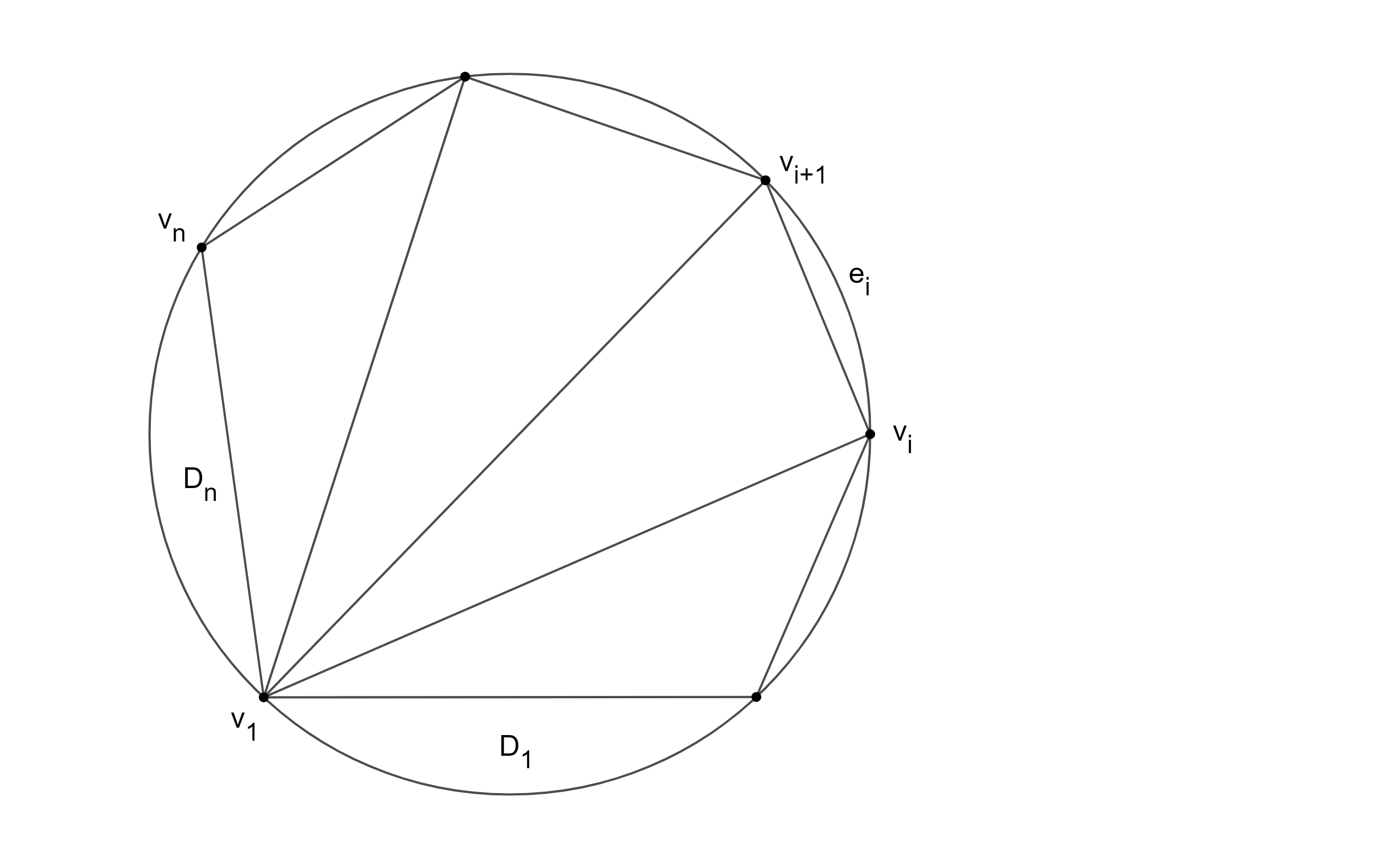}
    \caption{Decomposition of $D$.}
    \label{fig:decom_d}
\end{figure}

\proof
For $i=1,\ldots,n$, let $\alpha_i$ denote the loop composed of $\ph|_{e_i}$ and the geodesic in $X$ from $\ph(v_{i+1})$ to $\ph(v_{i})$.
To get $h$, we use Proposition~\ref{Pro:reshetnyak} to fill each $\alpha_i$ by a Lipschitz disc of area 
\[
a_i\leq\frac{1}{4\pi}\cdot\length_\ph(\al_i)^2\leq\frac{1}{\pi}\cdot\length_\ph(e_i)^2.
\]
The union of these discs defines $h$ and we have
\[
\area(h)\leq\sum_{i=1}^n a_i\leq\frac{1}{\pi}\cdot\length_\ph(\d D)^2\leq\frac{t^2}{\pi}.
\]
Furthermore, the Lipschitz disc $\psi$ satisfies
\[
\area(\psi)\leq\sum_{i=1}^n \area(T_i)\leq \sum_{i=1}^n \frac{t}{2}\cdot\length_\ph(e_i)\leq\frac{t^2}{2}.
\]
Hence, the Lipschitz $2$-sphere $\sigma$
formed by $\phi$ and $\psi$ together with $h$ has
\[
\area(\sigma)\leq\Bigl(a+\frac{1}{2}+\frac{1}{\pi}\Bigr)\cdot t^2<(a+1)\cdot t^2
\]
as required.
\qed
\medskip

Now we are ready for the proof of the implication $(3) \Rightarrow (2)$ in Theorem~\ref{main_1}.

\begin{Thm} \label{Thm:main_tech}
Let $X$ be a proper $\CAT(0)$ space of asymptotic rank at most~$2$. 
For every $\delta > 0$, there exists a constant $C$ such that every 
Lipschitz $2$-sphere $S\subset X$ extends to a Lipschitz  $3$-ball $B\subset X$ 
with volume 
\[
\vol(B) \le C \cdot\area(S)^{1+\del}.
\]
\end{Thm}

\begin{proof}
Suppose for contradiction that there exists $\delta\in(0,1/2)$
such that for every $M>0$ we find a Lipschitz $2$-sphere $f\colon S^2\to X$ with
\[
\Fillvol(f)\geq M\cdot \area(f)^{1+\delta}.
\]
Let $c:=c(\delta) \geq 1$ and $\eps:=\eps(\delta,1)\in(0,1]$ be the constants from Lemma~\ref{Lem:round-minimal} for $\theta = 1$.
We choose $f$ to be $(M,\delta,\eps)$-minimal, for some (large) $M$, to be determined later.
By the lemma, $f$ is intrinsically $(c,1)$-isoperimetric.
Since $\Fillvol(f) \le C \cdot \area(f)^{1 + 1/2}$ by Theorem~\ref{Thm:EII}, 
we have 
\[
\area(f)\geq \left(\frac{M}{C}\right)^{2/(1-2\delta)}.
\]
To arrive at a contradiction we will construct a filling of $f$ by a Lipschitz ball of small volume. To do this, we will homotop $f$ to a Lipschitz sphere consisting of a controlled number of minimal triangles, which we then fill using the linear isoperimetric inequality for minimal tetrahedra.

By Proposition~\ref{Pro:triangulation}, for $\area(f) \ge 4$ and $s>0$ with $\frac{1}{3c} \le s^2 \le \frac{1}{8c}\area(f)$, we find
 a Lipschitz $2$-sphere $f'\colon S^2\to X$ with $\area(f')\leq\area(f)$ and a Lipschitz homotopy $h$ from $f$ to $f'$ with
\[
\vol(h)\leq 12M\cdot c\cdot s^{2\delta}\cdot\area(f);
\]
moreover,  there exists a disc decomposition $S^2 = D_1 \cup \ldots \cup D_n$, induced by a graph $\Ga\subset S^2$, such that
\[
n \le \frac{14 \area(f)}{s^2}, \quad
\length_{f'}(\d D_i) \le 34 s, \quad  
\area_{f'}(D_i) \le 3c \cdot (34 s)^2
\]
for $i = 1,\ldots,n$. The value of $s$ will be specified later.
Applying Lemma~\ref{Lem:triangular_disc} to each of the Lipschitz discs $f'|_{D_i}$ and the 
corresponding set of vertices on $\d D_i$, we obtain Lipschitz $2$-spheres $\sigma_i$ with 
\[
\area(\sigma_i)\leq 4c \cdot (34 s)^2\leq \frac{3}{4}\cdot\area(f)\leq \area(f)-\eps,
\]
assuming that $6200c \cdot s^2\leq \area(f)$.
Thus, since $f$ is $(M,\delta,\eps)$-minimal,
\[
\Fillvol(\sigma_i)\leq M\cdot \area(\sigma_i)^{1+\delta}.
\]
We can arrange that the Lipschitz homotopies $\d D_i \times [0,1] \to X$ provided by 
Lemma~\ref{Lem:triangular_disc} agree on every common edge of two adjacent discs.
Filling each sphere $\sigma_i$ by a Lipschitz $3$-ball, we then obtain a Lipschitz homotopy $h'\colon S^2\times[0,1]\to X$
from $f'$ to a Lipschitz $2$-sphere $\bar f$ such that
\[
\vol(h') \leq \frac{14 \area(f)}{s^2}\cdot M\cdot (4\cdot 34^2)^{3/2} \cdot c^2 \cdot s^{2+2\delta}. 
\]
Combining $h$ and $h'$, we get a Lipschitz homotopy $\bar h$ from $f$ to $\bar f$ with volume
\[
\vol(\bar h) \le \vol(h) + \vol(h') \le 
5\cdot 10^6 c^2 \cdot M \cdot s^{2\delta} \cdot\area(f).
\]
By construction, the Lipschitz sphere $\bar f$ is piecewise minimal, 
that is, composed of finitely many minimal triangles spanned by geodesic triangles $T_1,\ldots,T_m$ in $X$, each with vertices in $f'(\d D_i)$ for some 
$i$ and, hence, circumference at most $34 s$. Moreover, every triangle $T_l$ has one side corresponding to an edge in $\Ga$, and so the number of triangles 
$m$ is bounded by twice the number $E$ of edges in $\Ga$. 
By Euler's formula, we have $E \leq 3n$ and therefore $m\leq 6n$.

Choose a point $o\in f'(\Ga)$. It follows from 
Proposition~\ref{Pro:triangulation} that 
\[
\diam(f'(\Ga))\leq\frac12 \cdot\length_{f'}(\Ga)\leq \frac{119\cdot\area(f)}{s}.
\]
Hence, by taking the geodesics from $o$ to the vertices of $T_l$ and inserting minimal triangles, we obtain a minimal tetrahedron $\tau_l$ with 
\[
\area(\tau_l)\leq 17 \cdot 119 \cdot \area(f) + (17s)^2 \leq 2024\cdot\area(f).
\]
(recall that we have added the assumption $6200c \cdot s^2 \le \area(f)$).
By Theorem~\ref{Thm:tetra-filling}, we have 
\[
\Fillvol(\tau_l)\leq 2024\mu\cdot \area(f).
\]
Again, since the number $m$ of tetrahedra is at most $6n$, we obtain a Lipschitz $3$-ball $\Phi\colon B^3\to X$ filling $\bar f$
with
\[
\vol(\Phi)\leq 2\cdot 10^5 \mu\cdot \frac{\area(f)^2}{s^2}.
\]
Taking the homotopy $\bar h$ into account, we get a Lipschitz $3$-ball $\Psi\colon B^3\to X$ filling $f$ with 
\begin{align*}
\vol(\Psi) &\leq \vol(\bar h)+\vol(\Phi) \\
&\leq 5\cdot 10^6 c^2 \cdot M \cdot s^{2\delta}\cdot \area(f)
+2 \cdot 10^5\mu\cdot\frac{\area(f)^2}{s^2}.
\end{align*}
Now we choose $s=\area(f)^{(1-\delta)/2}$ and arrive at
\[
\Fillvol(f)\leq \left(\frac{5\cdot 10^6 c^2}{\area(f)^{\delta^2}}+\frac{2\cdot 10^5\mu}{M}\right)\cdot M\cdot\area(f)^{1+\delta}.
\]
Recall that $\area(f) \ge (M/C)^{2/(1-2\del)}$. Thus, if 
\[
M > \max\left\{C\cdot(10^7 c^2)^{(1-2\del)/(2\del^2)}, 4\cdot 10^5 \mu \right\},
\]
then $\Fillvol(f) < M\cdot\area(f)^{1+\delta}$, in contradiction to 
the assumption.
\end{proof}

We conclude this section with 

\proof[Proof of Corollary~\ref{Cor:filling_surfaces}]
We fix $\delta>0$ and prove the claim by induction on $g$ where the case $g=0$ is the content of 
Theorem~\ref{main_1}. Now suppose we know the result for Lipschitz surfaces of genus at most $g-1$. We equip every closed surface $\Sigma$ with a Riemannian metric  
induced by an embedding $\iota:\Sigma\hookrightarrow\R^3$. 
Let $f \colon \Sigma_g \to X$ be a closed Lipschitz surface
of genus $g$.  For $\eps>0$ consider the related bilipschitz map
$f_\eps\colon\Sigma_g\to X\times\R^3$ with 
$f_\eps(x)=(f(x),\eps\iota(x))$. Choose $\eps$ small enough such that 
$\area(f_\eps)\leq 2\area(f)$. We will apply the systolic inequality for closed surfaces \cite{Gro_FRR, Gro_sys}. 
Since $\Sigma_g$ equipped with the induced path metric $d_{f_\eps}$ is bilipschitz to the Riemannian manifold $\Sigma_g$,
the systolic inequality holds for $(\Sigma_g,d_{f_\eps})$.
Thus, we find a non-contractible non-separating simple closed curve $\al\subset \Sigma_g$ such that 
\[\length_{f_\eps}(\al)\leq \sigma_g\cdot\sqrt{\area(f_\eps)}\]
where $\sigma_g$ denotes the systolic ratio.
By the isoperimetric inequality, we can fill $f_\eps\circ\alpha$ by a disc $D\subset X\times\R^3$ with $\area(D)\leq\frac{1}{4\pi}\cdot\length(f_\eps\circ\alpha)^2$. Inserting the disc $D$ cuts $\Sigma_g$
into a surface $\Sigma_\eps'$ of genus $g-1$ with 
\[\area_{f_\eps}(\Sigma_\eps')=\area_{f_\eps}(\Sigma)+\frac{\sigma_{g}^2}{2\pi}\cdot\area_{f_\eps}(\Sigma).\]
We denote by $\Sigma'$ the projection of $\Sigma'_\eps$ to $X$.
By induction, we obtain a Lipschitz handlebody 
$\bar f\colon H_{g-1}\to X$ filling $\Sigma'$ such that
\[\vol(H_{g-1})\leq C_{g-1}\cdot\area(\Sigma')^{1+\delta}.\]
Now $H_{g-1}$ can be Lipschitz parametrized by a genus $g$ handlebody $H_g$ filling $\Sigma_g$. We obtain
\begin{align*}
\vol(H_{g})&=\vol(H_{g-1})\leq C_{g-1}\cdot\area(\Sigma')^{1+\delta}\\
&\leq C_{g-1}\cdot\area(\Sigma_\eps')^{1+\delta}\\
&\leq C_{g-1} \biggl( 2+\frac{\sigma_{g}^2}{\pi} \biggr)^{1+\delta}\cdot\area(\Sigma)^{1+\delta}.
\end{align*}
This proves the claim with $C_g:=C_{g-1}\bigl(2+\frac{\sigma_{g}^2}{\pi}\bigr)^{1+\delta}$.
\qed

% %%%%%%%%%%%%%%%%%%%%%%%%%%%%%%%%%%%%%%%%%%%%%%%%%%%%%%%%%%%%%%%%%%%%%%%%%%%%%%%

\bibliographystyle{alpha}
\bibliography{gap}

%%%%%%%%%%%%%%%%%%%%%%%%%%%%%%%%%%%%%%%%%%%%%%%%%%%%%%%%%%%%%%%%%%%%%%%%%%%%%%%

\end{document}